\documentclass[12pt]{article} 
\usepackage{amsmath, amsfonts, amssymb}
\usepackage{mathrsfs}
\usepackage{theorem}
\usepackage{bm}
\usepackage{graphicx}
\usepackage{subcaption}
\usepackage{float}
\usepackage{bm}
\RequirePackage[colorlinks,citecolor=blue,urlcolor=blue,bookmarksopen]{hyperref}
\usepackage[utf8]{inputenc} 
\usepackage[T1]{fontenc}
\newtheorem{mythm}{Theorem}[section]
\newtheorem{myprop}[mythm]{Proposition}
\newtheorem{mylem}[mythm]{Lemma}

{\newtheorem{myrem}[mythm]{Remark}}
{\newtheorem{myexam}{Example}[section]}

\def\R{\mathbb R}

\def\N{\mathbb N}
\def\B{\mathscr B}

\def\F{\mathscr F}

\def\C{\mathscr C}
\def\d{\text{\rm{d}}}
\def\E{\mathbb E}
\def\p{\mathbb P}

\def\la{\langle}
\def\raa{\rangle}
\def\La{\Lambda}

\def\e{\mathrm{e}}

\def\D{\mathscr D}

\def\S{\mathcal S}

\def\pb{\mathscr{P}}

\def\wt{\widetilde}

\def\W{\mathbb{W}}

\newenvironment{proof}{{\noindent\it Proof.}\ }{\hfill $\square$\par}

\numberwithin{equation}{section}

\allowdisplaybreaks

\begin{document}
	
	\title{Invariant probability measures for path-dependent regime-switching processes}
	
	\author{Jianrui Li$^{a}$, \   \ Jinghai Shao$^{b}$%
		\thanks{a: School of Mathematical Sciences, Tiangong University, Tianjin 300387, China. \\
			\hspace*{1.52em}
			b: Center for Applied Mathematics, Tianjin University, Tianjin 300072, China. \\
			\  Emails: jianruili@tju.edu.cn (Li); shaojh@tju.edu.cn (Shao)
	}}
	
	\maketitle
	
\begin{abstract}
We consider stochastic functional differential equations with Markovian regime-switching on an (infinitely) countable state space. Via establishing the Feller property and contraction in $L_1$-Wasserstein distance of the semigroup associated with the segment process, we show the existence and uniqueness of  invariant probability measure. In order to cope with switching processes on an infinite state space, we develop a truncation method based on a new construction of order preserving coupling for Markov chains to generalize the estimation of exponential functionals for Markov chains with infinite states. The existing sharp estimation is only valid for Markov chains on a finite state space, as it fundamentally relies on the Perron-Frobenius theorem.
	\end{abstract}
	
	

	\section{Introduction}

	This work concerns the existence and uniqueness of invariant probability measure for the stochastic functional differential equation (SFDE) with Markovian regime-switching. Let $\tau_0>0$.  $\C=C([-\tau_0,0];\R^d)$ is the collection of continuous functions from $[-\tau_0,0]$ to $\R^d$, and  $\D=\D([-\tau_0,0];\S)$ is the set of functions $g:[-\tau_0,0]\to \S$ which are right continuous with left  limits.  Consider the following SFDE:
	\begin{equation}\label{a-1}
		\d X(t)=b(X_t,\La_t)\d t+\sigma(X_t,\La_t)\d W(t),\quad t>0 
	\end{equation}
	with initial value $X_0=\xi\in \C$, $\La_0=\ell\in \D$, where $b:\C\times\D\to \R^d$, $\sigma:\C\times\D\to \R^{d\times d}$ are measurable functions, $(W(t))$ is a $d$-dimension Brownian motion, $(\La(t))$ is a continuous-time Markov chain on a discrete state space $\S=\{1,2,\ldots,N\}$ with $N\leq \infty$.  The segment process $X_t\in \C$ and $\La_t\in \D$ are defined as 
	\begin{equation}\label{a-3}
		X_t(r)=X(t+r),  \quad \La_t(r)=\La( t+r ),\quad  r\in [-\tau_0,0], \ t\geq 0.
	\end{equation} 
	The evolution of $X(t)$ depends on the behavior of the process during the interval $[t-\tau_0, t]$. Consequently, $X(t)$ no longer possesses the Markov property. However, if a segment  of the process of length $\tau_0$ is extracted  and treated as a process on the infinite dimensional path space, the augmented process $X_t$ defined above exhibits the Markov property; see the details in  Proposition \ref{prop-2.2} below. 
	
	There is a rather comprehensive literature on the  stability and ergodicity of SFDEs with regime-switching in the form
	\begin{equation}\label{a-2}
		\d X(t)=b(X_t,\La(t))\d t+\sigma(X_t,\La(t))\d W(t), \quad t>0,
	\end{equation}
	with initial value $(X_0,\La(0))=(\xi,i)\in \C\times\S$. The inclusion of segment processes $(X_t,\La_t)$ rather than the processes $(X_t,\La(t))$ in the equation \eqref{a-1} renders this process a suitable model for a broader range of applications.
	For instance, in the study of stabilization of regime-switching processes, \cite{Mao} raised an interesting method to stabilize the studied system based on discrete time observations
	\[\d X(t)=\big(b(X(t),\La(t))-u(X([t/\tau_0]\tau_0),\La(t))\big)\d t+\sigma(X(t),\La(t))\d B(t)\]
	in order to ease operation and reduce cost. This is in the form of SFDE \eqref{a-2}. To enhance practical operability, it is necessary to use the discrete-time observations of $(\La(t))$ to stabilize the studied system. As being done in e.g. \cite{Sh17,SZLM}, one needs to study the stability of the following SFDE:
	\[\d X(t)=\big(b(X(t),\La(t))-u(X([t/\tau_0]\tau_0),\La([t/\tau_0]\tau_0))\big)\d t+\sigma(X(t),\La(t))\d B(t),\]
	which turns to be in the form of SDE \eqref{a-1}.  This underscores the necessity of extending the framework from SFDE \eqref{a-2} to \eqref{a-1}.

	The generalization from \eqref{a-2} to \eqref{a-1} introduces significant mathematical challenges in handling (as detailed in next paragraph).    There is limited work on the ergodicity of SFDE \eqref{a-1}.  For instance, \cite{BSY23} and \cite{NNN} studied the existence and uniqueness of invariant probability measure for the process $(X_t,\La(t))$ and its Euler-Maruyama's approximation. Such investigation has been extended in \cite{SWW,WWY,WWZ} to the situation with infinite delay together with further investigation on limit theorem and large deviations. The aforementioned works all considered the case of a finite state space $\S$. This restriction is essentially due to the application of the Perron-Frobenius theorem to estimate the expectation of   exponential functionals of continuous-time Markov chains established in \cite{BGM}. This work shall overcome the well-known challenge on the finiteness of $\S$ in the study of long time behavior of stochastic processes with regime-switching. 
	
	In contrast to \eqref{a-2}, to study the SFDE \eqref{a-1} including the delay term of the switching process $(\La(t))$ is quite necessary. For instance, inspired by the work \cite{Mao}, many works studied the stabilization problem of an equilibrium point for SDEs with regime-switching by feedback control based on discrete-time observations of $(X(t))$ and $(\La(t))$ such as \cite{Sh17,SZLM,YLL} and references therein.   \cite{LMMY} studied the stabilization problem by delay feedback control based on discrete-time observations of $(X(t))$ and $(\La(t))$. In addition, the stability in distribution for a stochastic system is an important topic in many application fields (cf. e.g. \cite{BB92,YM03,YZM}).  Recently, together with Zhao, we have studied in \cite{LSZ} the stability in distribution for regime-switching diffusion processes with feedback control based on discrete-time observations. Two methods were used in \cite{LSZ}, and both of them depend highly on the special structure of the discrete-time observations. In this work, we  shall further advance the research on this issue by considering more general cases.

	According to the study on the ergodicity of SFDEs such as in \cite{BM05,BS17,EGS}, a crucial condition to ensure the existence of invariant probability measure is that the drift coefficient must be dissipative w.r.t. the current state. A little precisely,
	for the SFDE
	\[\d X(t)=f(X_t)\d t+g(X_t)\d W(t),\]
	the drift should satisfy 
	\[\la \xi(0), f(\xi)\raa \leq -c |\xi(0)|^p,\quad \xi\in \C \]
	for some $p\in (0,1]$, $c>0$. In the presence of regime-switching, it is not necessary for the drift coefficients to be dissipative in every fixed environment, which has already been observed in \cite{BSY23,NNN} via concrete examples. However, based on the results of this work, we observe that: the delay term $\La_t$ in the drift $b(\cdot,\cdot)$ differs fundamentally in its effect from the delay term $X_t$ in $b(\cdot,\cdot)$; see condition (H1) and Remark \ref{rem-1} (ii). Roughly speaking,  the drift $b(\cdot,\cdot)$ needs to be dissipative in the current state $X(t)$ of the segment process $(X_t)$, but the average principle dominates the role of the process $(\La(t))$ in the study of invariant probability measure of $(X_t,\La_t)$. 
	
	Compared to the existing literature on SFDEs with regime-switching, we shall overcome two main difficulties. First, in view of the fact that the coefficients $b$ and $\sigma$ depend on $\La_t\in \D$ instead of $\La(t)\in \S$, only the segment process $(X_t,\La_t)$ on the path space owns the Markov property, not on the finite dimensional space, and the complexity of the infinite dimensional space $\C\times\D$ significantly increases the difficulty of this problem. In particular, the complexity of the topology of $\D$ makes it difficult to study the Feller property of the corresponding semigroup $P_t$ and its convergence. 
	
	Second, we shall overcome the restriction on the finiteness of  the state space $\S$ and study the invariant probability measure for SFDEs with regime-switching on an infinite state space. We truncate a Markov chain with infinite states into a new one possessing only finite states through constructing a suitable coupling process in terms of a Poisson random measure. This helps us to develop the estimates on the exponential functionals of Markov chains on a finite state space established in \cite{BGM} to the infinite state space case. The estimation obtained in \cite{BGM} is sharp since its lower and upper bounds differ only by a constant factor. The method of \cite{BGM} depends on the Perron-Frobenius theorem, which is only valid for matrices of finite order. 
	
	Moreover, this work not only shows that the existence of the unique invariant probability measure for the segment process, but also provides explicit estimate of the convergence rate in the $L_1$-Wasserstein distance. In view of the sharp estimate in Lemma \ref{lem-2.4} for the exponential functionals of Markov  chains, it is quite reasonable to introduce the largest eigenvalue  in Theorem \ref{main-1} to control the convergence rate.   
	
	This work is organized as follows. In Section 2, we present the basic set up of our model, and introduce the assumptions on the coefficients and the fundamental properties of the solutions to SFDE \eqref{a-1}. Section 3 studies the existence and uniqueness of invariant probability measure for SFDE \eqref{a-1} with regime-switching on a finite state space. SFDEs \eqref{a-1} with regime-switching on an infinite state space are studied in Section 4.

	\section{Framework and basic properties}
	
	In this section, we shall introduce the fundamental setting of this work, and present the conditions on the coefficients of the studied SFDE  \eqref{a-1} with regime-switching. We want to provide the explicit estimate of the convergence rate of the segment process to its invariant probability measure in this work. Hence, the well studied coupling method for the continuous-time Markov chains will be introduced. 
	\subsection{Basic setup}
	Let us introduce some notations used below. Let $\S=\{1,2,\ldots, N\}$ with $N\leq \infty$. Let $\tau_0>0$ be fixed in this work. $\C=C([-\tau_0,0];\R^d)$ denotes the collection of continuous functions $f:[-\tau_0,0]\to \R^d$, which is endowed with the norm 
	\[\|f\|_\infty=\sup_{t\in [-\tau_0,0]} |f(t)|.\]
	Let $\D=\D([-\tau_0,0];\S)$ be the set of functions $g:[-\tau_0,0]\to \S$ which are right continuous with left  limits. 
	
	
	Let $(\Omega_1,\F_1,\p_1)$ and $(\Omega_2,\F_2,\p_2)$ be two probability spaces. 
	Put
	\begin{equation}\label{b-0}
		(\Omega,\F,\p)=\big(\Omega_1\times\Omega_2, \F_1\times\F_2,\p_1\times \p_2\big).
	\end{equation}
	Let $(W(t))$ be a $d$-dimension standard Brownian motion on $\Omega_1$, and $(\La(t))$ is a continuous-time Markov chain on $\Omega_2$. Then, under the probability measure $\p=\p_1\times\p_2$, two processes $(W(t))$ and $(\La(t))$ are mutually independent. 
	
	Recall the SFDE \eqref{a-1} with Markovian regime-switching to be studied in the sequel:
	\begin{equation*}\label{b-1}
		\d X(t)=b(X_t,\La_t)\d t+\sigma(X_t,\La_t)\d W(t), \quad t>0,
	\end{equation*}
	with initial value $(X_0,\La_0)=(\xi,\ell)\in \C\times\D$. 
	In the existing literature, $(X_t)_{t\geq 0}$ is usually called the segment process associated with the SFDE \eqref{a-1}. $(\La(t))$ is a continuous-time Markov chain on the countable state space $\S$, whose transition rate matrix is given by $Q=(q_{ij})_{i,j\in \S}$. Therefore, this process $\La(t)$ for $t>0$ is only determined by $\La(0)=\ell(0)$ independent of $\ell(r)$ for $r\in [-\tau_0,0)$. 
	

	
	In order to provide explicit estimate on the convergence rate of the process $(\La(t))$ to its stationary distribution, we would like to use the coupling method, which has been extensively studied (cf. e.g. \cite{Chen}). Let $\wt Q$ be a coupling operator of $Q$ and itself, that is, $\wt Q$ is an operator on $\S\times \S$ satisfying
	\[\wt Q f(i,j)= Q f(\cdot,  j)(i)+ Q f(i, \cdot)(j),\quad f\in \B(\S\times\S), \]
	where $\wt Q f(i,j)=\sum_{(k,m)\in \S\times\S} \tilde q( i,j; k,m )\big(f(k,m)-f(i,j)\big)$ and $Qg(i)=\sum_{j\in\S} q_{ij}(g(j)-g(i))$ for $g\in \B(\S)$. 
	
	Here we provide two examples of the coupling operators $\wt Q$.
	\begin{myexam}\label{coup-1}
		Define $\wt Q=(\tilde q( i,j; m,n ))$ as follows:
		\begin{equation*}
			\begin{cases}
				\tilde q(i,j;m,n)=(q_{ik}-q_{jk})^+, &m=k,n=j,k\neq i,\\
				\tilde q(i,j;m,n)=(q_{jk}-q_{ik})^+, &m=i,n=k,k\neq j,\\
				\tilde q(i,j;m,n)=q_{ik}\wedge q_{jk}, &m=k, n=k,(k,k)\neq (i,j),\\
				\tilde q(i,j;m,n)=0, &\text{others, if $(m,n)\neq (i,j)$},
			\end{cases}
		\end{equation*}
		and $\tilde q(i,j;i,j)=-\sum_{(m,n)\neq (i,j)} \tilde q(i,j;m,n)$. 
		This coupling is called the basic coupling of $Q$ and itself. 
	\end{myexam}
	
	\begin{myexam}\label{coup-2}
		The classical coupling $\wt Q$ of $Q$ and itself is defined by 
		\begin{equation*}
			\begin{cases}
				\tilde q(i,j;m,n)=q_{im} ,&i\neq j,m\neq i,n= j,\\
				\tilde q(i,j;m,n)=q_{jn}, &i\neq j,m=i,n\neq j,\\
				\tilde q(i,j;m,n)=q_{im}, &i=j, m=n\neq i,\\
				\tilde q(i,j;m,n)=0, & \text{others, if $(m,n)\neq (i,j)$},
			\end{cases}
		\end{equation*}and $\tilde q(i,j;i,j)=-\sum_{(m,n)\neq (i,j)} \tilde q(i,j;m,n)$. 
	\end{myexam}
	
	Associated with the coupling operator $\wt Q$, there is a coupling process $(\La(t),\La'(t))$ on $\S\times \S$. By the definitions of the basic coupling and the classical coupling we can see that once the process $(\La(t))$ meets $(\La'(t))$ then they will move together. Precisely, let 
	\begin{equation}\label{b-2} 
		T=\inf\{t\geq 0; \La(t)=\La'(t)\},
	\end{equation}
	which is called \emph{the coupling time} of $(\La(t),\La'(t))$. For the classical coupling and the basic coupling, if $T<\infty$, then $ \La(t)=\La'(t)$ for $t\geq T$. The couplings of Markov chains have been extensively studied, and we refer the readers to the monograph \cite{Chen} for more introduction on this topic. 
	
	The conditions on the process $(\La(t))$ are collected below. 
	\begin{itemize}
		\item[(A1)] $Q=(q_{ij})$ is an irreducible, conservative transition rate matrix satisfying
		\[K_0:=\sup_{i\in\S} q_i<\infty, \quad \text{where $q_i:=-q_{ii}=\sum_{j\neq i} q_{ij}$}.\]
		\item[(A2)] 
		Let $\wt Q$ be a coupling of $Q$ and itself such that the associated process $(\La(t),\La'(t))$ moves together after the coupling time.  There exists a function $V:\S\times \S\to [0,\infty)$ such that $\|V\|_{\infty}:=\sup_{i,j\in\S} V(i,j) <\infty$,
		\begin{equation}\label{b-3} 
			\wt Q V(i,j)\leq -1,\quad i,j\in\S,\ i\neq j.
		\end{equation}
	\end{itemize} 
	\begin{myrem}\label{rem-0}
		Condition (A1) ensures the existence of the Markov process $(\La(t))$ with transition rate matrix $Q=(q_{ij})$. As we shall study the ergodicity of SFDE with regime-switching $(X(t),\La(t))$, it is natural  to assume first that $(\La(t))$ itself is ergodic. Condition (A2) can ensure that the Markov chain $(\La(t))$ is ergodic, and can provide explicit estimate of the convergence rate via the estimate
		of $\E \e^{\lambda T}$ for some $\lambda >0$ for the coupling time $T$. This will be used to provide explicit estimate of the convergence rate of $(X_t,\La_t)$ to its invariant probability measure. 
	\end{myrem}

	To be more precise, 
	when $\S$ is a finite state space,  we can take $\wt Q$ to be the classical coupling given in Example \ref{coup-2}, then the associated coupling process $(\La(t),\La'(t))$ must be  ergodic under (A1) and hence there is $\lambda>0$ such that 
	$\E[ \e^{\lambda T}]<\infty$. 
	Refer to  \cite{And} or \cite[Chapter 5]{Chen05} for more details. 
	
	Now let us explain how to get the estimate of $\E[\e^{\lambda T}]$ via \eqref{b-3}. Indeed, by \eqref{b-3}, for any $0<\lambda<\|V\|_\infty^{-1}$,
	\begin{align*}
		\E\Big[ \int_0^{t\wedge T}\!\e^{\lambda s}\d s\Big]&\leq V(i,j)+\E\Big[ \int_0^{t\wedge T} \lambda \e^{\lambda s} V(\La(s), \La'(s))\d s\Big],\\
		\lambda^{-1}\E\Big[\e^{\lambda(t\wedge T)}-1\Big]&\leq V(i,j)+\|V\|_{\infty} \E\Big[\e^{\lambda(t\wedge T)}-1\Big]. 
	\end{align*}
	Letting $t\to\infty$, we obtain that 
	\begin{equation}\label{b-4}
		\E\big[\e^{\lambda T}\big]\leq \frac{1}{1-\lambda \|V\|_\infty}<\infty.
	\end{equation}

	Let us introduce the Skorokhod metric on $\D$. The discrete metric on $\S$ is denoted by $\vartheta(i,j)$, i.e. $\vartheta(i,j)=1$ if $i\neq j$; $\vartheta(i,j)=0$ if $i=j$. Let 
	\begin{equation}\label{Theta}
		\begin{split}
			\Theta=\Big\{\theta:[-\tau_0,0]\to [-\tau_0,0]; \ &\text{increasing Lipschitz continuous functions satisfying}\\
			&  \gamma(\theta):=\sup_{0\geq s>t\geq -\tau_0} \Big|\log \frac{\theta(s)-\theta(t)}{s-t}\Big|<\infty\Big\}.
		\end{split}
	\end{equation}
	For any $\ell,\tilde \ell\in \D$, define
	\begin{equation}\label{b-5}
		\rho(\ell,\tilde \ell)=\inf_{\theta\in \Theta} \max\big\{\gamma(\theta), \sup_{-\tau_0\leq t\leq 0} \vartheta\big(\ell(t), \tilde \ell(\theta(t))\big)\big\}.
	\end{equation}
	As $\vartheta(i,j)\leq 1$, it is clear that $\rho(\ell,\tilde \ell)\leq 1$ for any $\ell,\tilde \ell\in \D$ by the definition \eqref{b-5}. 
	According to \cite[Chapter 3, Proposition 5.3]{EK86}, for $\ell,\ell_m\in \D$, $\lim_{m\to\infty} \rho(\ell_m,\ell)=0$ if and only if there exist $\theta_m\in \Theta$, $m\geq 1$, such that 
	$\lim_{m\to \infty} \gamma(\theta_m)=0$ and 
	\[\lim_{m\to\infty} \sup_{-\tau_0\leq t\leq 0} \vartheta(\ell_m(t), \ell(\theta_m(t)))=0.
	\] 
	Due to \cite[Chapter 3, Theorem 5.6]{EK86}, $(\D,\rho)$ is a complete metric space.

	\subsection{Wellposedness and the Markov property of the segment processes}
	
	In this part, we shall provide conditions on the coefficients of equation \eqref{a-1} to ensure the wellposedness of the SFDE \eqref{a-1} with regime-switching. Then, we show that the associated segment process $(X_t,\La_t)$ will admit Markov property.
	
	Let us define a metric $d$ on $\C\times \D$ by 
	\begin{equation}\label{b-6}
		d\big((\xi,\ell),(\eta,\tilde \ell )\big)=\|\xi-\eta\|_{\infty} +\rho(\ell, \tilde \ell),\quad (\xi,\ell), (\eta,\tilde \ell)\in \C\times\D.
	\end{equation}
	Then $(\C\times\D, d)$ is a complete metric space.
	Let $\pb(\C\times\D)$ be the space of all probability measures over $\C\times \D$. Let 
	\[\pb_{p,d}(\C\times \D)=\Big\{ \nu\in \pb(\C\times\D);\ \int_{\C\times\D} d\big((\xi,\ell),(o,\ell')\big)^p\nu(\d \xi, \d \ell)<\infty\Big\},\quad p>0,
	\] where $(o,\ell')$ is some fixed point in $\C\times\D$.  
	Correspondingly, we introduce the $L_1$-Wasserstein distance on $\pb(\C\times\D)$: for any $\nu_1,\,\nu_2\in \pb(\C\times\D)$, 
	\[\W_{1}(\nu_1,\nu_2)=\inf_{\pi\in \mathcal{C}(\nu_1,\nu_2)} \Big\{ \int_{(\C\times\D)^2}\!\! d\big((\xi,\ell),(\eta,\tilde \ell)\big)\pi\big((\d \xi,\d \ell),(\d \eta,\d\tilde \ell)\big)\Big\},
	\]
	where $\mathcal{C}(\nu_1,\nu_2)$ stands for the collection of all probability measures on $(\C\times\D)^2$ with marginals $\nu_1$ and $\nu_2$ respectively. Namely, $\pi\in \mathcal{C}(\nu_1,\nu_2)$ if and only if $\pi(\cdot, \C\times\D)=\nu_1(\cdot)$ and $\pi(\C\times \D,\cdot)=\nu_2(\cdot)$. According to the Kantorovich dual representation of the Wasserstein distance $\W_1$ (cf. \cite{Vil}), it holds that 
	\begin{equation}\label{b-7}
		\W_1(\nu_1,\nu_2)=\sup_{F\in \mathrm{Lip}_1(\C\times\D)} \Big\{\int_{\C\times\D} F(\xi,\ell)\nu_1(\d \xi,\d \ell)-\int_{\C\times\D} F(\xi,\ell)\nu_2(\d \xi,\d \ell)
		\Big\},
	\end{equation}
	where $\mathrm{Lip}_1(\C\times\D)$ denotes the set of Lipschitz continuous functions $F$ on $\C\times \D$ satisfying 
	\[\sup_{(\xi,\ell)\neq (\eta,\tilde \ell)} \frac{|F(\xi,\ell)-F(\eta, \tilde \ell)|}{d \big( (\xi,\ell), (\eta, \tilde \ell)\big)}\leq 1.
	\]
	
	Now let us introduce the conditions on the coefficients of the process $(X(t))$.  
	\begin{itemize}
		\item[(H1)] There are constants $\kappa_1>0$, $\epsilon\in [-\tau_0,0]$, $\alpha_i\in \R$, $\beta_i\in \R_+\!:=[0,\infty)$ for $i\in \S$, and a probability measure $\mu$ on $[-\tau_0,0]$ such that 
		\begin{equation}\label{c-1}
			\begin{split}
				&2\la \xi(0)-\eta(0), b(\xi,\ell)-b(\eta,\ell)\raa+ \|\sigma(\xi, \ell)-\sigma(\eta,\ell)\|_{\mathrm{HS}}^2\\
				&\leq \alpha_{\ell(\epsilon)} |\xi(0)-\eta(0)|^2+ \beta_{\ell(\epsilon)}\int_{-\tau_0}^0 \!|\xi(s)-\eta(s)|^2\mu(\d s),\\       
				\|\sigma(\xi,\ell)-& \sigma(\eta,\ell)\|_{\mathrm{HS}}^2\leq \kappa_1 \Big( |\xi(0)-\eta(0)|^2+\int_{-\tau_0}^0|\xi(s)-\eta(s)|^2\mu(\d s) \Big) 
			\end{split}
		\end{equation}
		for $\xi,\eta\in \C,\, \ell\in\D$, where $\|\,\cdot\,\|_{\mathrm{HS}}$ denotes the Hilbert-Schmidt norm. 
		
		\item[(H2)] There is a constant  $\kappa_2>0$ such that 
		\begin{equation}\label{c-2}
			|b(\xi,\ell)|^2 +\! \|\sigma(\xi,\ell)\|_{\mathrm{HS}}^2 \! \leq \kappa_2 \Big(|\xi(0)|^2\!+\! \int_{-\tau_0}^0 \!\! |\xi(s)|^2\mu(\d s)\Big),\quad \xi\in \C,\, \ell\in \D,
		\end{equation}
		where $\mu$ is given in (H1).
		
		\item[(H3)] There is a constant $\kappa_3>0$ such that 
		\[|b(\xi,\ell)-b(\xi, \tilde \ell)|^2+\|\sigma(\xi,\ell)-\sigma(\xi,\tilde \ell)\|_{\mathrm{HS}}^2\leq \kappa_3(1+\|\xi\|_\infty^2) \rho(\ell,\tilde \ell)^2\]
		for $\xi\in \C$, $\ell,\,\tilde\ell\in \D$. 
	\end{itemize}
	
	\begin{myrem}\label{rem-1}
		$\mathrm{(i)}$ Condition (H1) is used to ensure the ergodicity of $(X(t))$.  For SFDEs without regime-switching, it usually demands the coefficient $\alpha$ of the term $ |\xi(0)-\eta(0)|^2$ to be negative to ensure the existence of invariant probability measures, however, the term $\beta_i\int_{-\tau_0}^0\! |\xi(s)-\eta(s)|^2\mu(\d s)$ is viewed as certain small perturbations. See, for example, \cite{BM05,BS17,EGS}. In the context, we do not require all parameters $\alpha_i$ to be negative; rather, we allow some to be positive and others to be negative. This reflects the effect of random regime-switching.   Note that requiring $\beta_i\geq 0$ is not a restrictive assumption. If (H1) holds for some $\beta_{i}<0$, we can simply set $\beta_i=0$ to ensure (H1) remains valid.  
		
		$\mathrm{(ii)}$ In \eqref{c-1}, we use $\alpha_{\ell(\epsilon)}$ for some $\epsilon\in [-\tau_0,0]$ to emphasize the different role played by $\xi-\eta\in \C$ and $\ell\in \D$. In view of (i), the right hand side of \eqref{c-1} depending on $\ell(0)$ or $\ell(-\tau_0)$ will cause no essential impact to our main results; see Theorem \ref{main-1} below. So, $\mathrm{(H1)}$ can be generalized to more general form such as $\alpha(\ell(0), \ell(-\tau_0),\ell(\epsilon))$ for some function $\alpha:\S\times\S\times \S\to \R$ without causing essential difference in the deduction of our main results.
	\end{myrem}
	
	\begin{myrem}\label{rem-1.5}
		Condition (H2) is imposed to ensure the nonexplosion of the process $(X(t))$. Condition (H3) will be used in the study of the Feller property of the segment process $(X_t,\La_t)$ in Lemma \ref{lem-2.3}.
	\end{myrem}

	First, let us present a result on the wellposedness of the SFDE \eqref{a-1}. 
	\begin{myprop}\label{prop-2.1}
		Assume $\mathrm{(A1)},\,\mathrm{(A2)}$, $\mathrm{(H1)}$-$\mathrm{(H3)}$ hold. For any  $\ell\in \D$, let $(\La(t))_{t\geq 0}$ be a continuous-time Markov chain on $\S$ with transition rate matrix $Q$ and initial value $\La(0)=\ell(0)$. Then, for any $ \xi \in \C $, SFDE \eqref{a-1} admits a unique strong solution $(X(t),\La(t))_{t\geq 0}$ with initial value $(X_0,\La_0)=(\xi,\ell)$.
	\end{myprop}
	
	\begin{proof}
		This proposition can be proved in a standard way in the study of SDEs with Markovian regime-switching (cf. \cite{MY}, \cite[Lemma 2.2]{SZ21}). According to \cite[Theorem 2.3]{RS10}, for each fixed $\tilde \ell\in \D$, under the conditions (H1) and (H2), the following SDE 
		\[\d Y(t)=b(Y_t,\tilde \ell)\d t+\sigma(Y_t,\tilde \ell)\d W(t),\quad Y_0=\xi 
		\] admits a unique solution. The  method in \cite{RS10} is Euler-Maruyama's approximation, which can be easily extended to the situation of time-inhomogeneous SFDEs. We shall use the time-inhomogeneous version of \cite[Theorem 2.3]{RS10} in the following to prove our result by taking advantage of the independence between $(\La(t))_{t\geq 0}$ and $(W(t))_{t\geq 0}$ step by step. 
		
		Let 
		\[0\leq \tau_1<\tau_2<\cdots<\tau_n<\cdots\]
		be the sequence of jumping times of the Markov chain $(\La(t))_{t\geq 0}$. Due to (A1),  we have $\lim_{n\to\infty }\tau_n=\infty$ almost surely. 
		For $t\in [0,\tau_1)$, consider the SDE
		\[\d Y^{(1)}(t)=b(Y_t^{(1)}, \ell_t^{(1)})\d t+\sigma(Y_t^{(1)},\ell^{(1)}_t)\d W(t),
		\quad Y_0^{(1)}=\xi,
		\]
		where, for $\theta\!\in \! [-\tau_0,0]$, $\ell_t^{(1)}(\theta)\!=\!\ell^{(1)}( t+\theta )=\ell(t+\theta)$ if $t+\theta\leq \!0$; $\ell^{(1)}_t\!(\theta)=\ell^{(1)}\!(t+\theta) \!=\! \ell(0)$ if $t+\theta\in [0,\tau_1)$. The solution $ Y^{(1)}(t) $  exists uniquely due to \cite[Theorem 2.3]{RS10} for $t\in [0,\tau_1)$. Then, let $Y^{(1)}(\tau_1)=\lim_{t\uparrow\tau_1}Y^{(1)}(t)$. So, $Y_{\tau_1}^{(1)}$ is well defined taking values in $\C$. 
		
		Next, for $t\in [\tau_1,\tau_2)$, let us consider the SDE
		\[\d Y^{(2)}(t)=b(Y_t^{(2)},\ell^{(2)}_t)\d t+\sigma(Y_t^{(2)},\ell_t^{(2)} )\d W(t),\quad Y_{\tau_1}^{(2)}= Y_{\tau_1}^{(1)},\]
		where, for $\theta\in [-\tau_0,0]$,  $\ell^{(2)}_t(\theta)=\La(\tau_1)$ if $t+\theta\in [\tau_1,\tau_2)$; $\ell^{(2)}_t(\theta)=\ell^{(1)}(t+\theta)$ if $t+\theta<\tau_1$. The solution $Y^{(2)}(t)$ exists uniquely for $t\in [\tau_1,\tau_2)$ by \cite[Theorem 2.3]{RS10} again. Put $Y^{(2)}(\tau_2)=\lim_{t\uparrow \tau_2} Y^{(2)}(t)$. 
		
		Similar to the construction of solutions to SDEs with Markovian regime-switching as \cite[Theorem 3.8]{MY},  the desired solution can be obtained through successive construction.   
	\end{proof}
	
	The dependence of the coefficients of SFDE \eqref{a-1} on the historical trajectory  causes the solution $(X(t),\La(t))\in \R^d\times\S$ lose Markov property. One has to consider the segment process $(X_t,\La_t)$ in the infinite dimensional space $\C\times \D$ to restore the Markov property.  
	For $(\xi,\ell)\in \C\times \D$, let $(X^{\xi,\ell}(t), \La^{\ell}(t))$ be the solution to SFDE \eqref{a-1} with initial value $X^{\xi,\ell}_0=\xi\in \C$ and $\La^\ell_0=\ell \in \D$.

	\begin{myprop}\label{prop-2.2}
		Assume $\mathrm{(A1)},\,\mathrm{(A2)}$, $\mathrm{(H1)}$-$\mathrm{(H3)}$ hold.  Then the  process $(X^{\xi,\ell}_t,\La^\ell_t)_{t\geq 0}$ is a Markov process on $\C\times\D$. 
	\end{myprop}
	
	\begin{proof}
		We shall apply the result of \cite[Proposition 3.3]{RRG} on the Markov property of segment process associated with SFDEs  driven by L\'evy processes. To this aim, we only need to rewrite $(\La(t))$ as a stochastic integration w.r.t. the Poisson random measure via Skorokhod's representation theorem, which has been extensively studied in the study of regime-switching processes; see, e.g. \cite{YZ10} and its recent application  to the comparison theorem in \cite{Sh24}.
		
		Define a sequence of intervals $\Gamma_{ij}$, $i\neq j\in \S$, as follows: for each $i\in \S$, 
		\begin{align*}
			\Gamma_{i1}&=[0, q_{i1}),\ \Gamma_{i2}=[q_{i1}, q_{i1}+q_{i2}),\ldots, 
			\Gamma_{iN}=\Big[\sum_{j=1, j\neq i  }^{N-1}q_{ij} , \sum_{j=1, j\neq i  }^N q_{ij}\Big).
		\end{align*}
		By (A1), $\sum_{j=1,j\neq i}^N q_{ij}=q_i<\infty$, and hence each interval $\Gamma_{ij}$ with $i\neq j$ is a finite interval. For the convenience of notation, we put $\Gamma_{ii}=\emptyset$. Let $\mathcal{N}(\d t, \d z)$ be a Poisson random measure on $[0,\infty)\times \R_+$ with intensity measure $\d t\times\bm{m}(\d z)$ and independent of $(W(t))$, where $\bm{m}(\d z)$ denotes the Lebesgue measure over $\R_+$. 
		Let 
		\[h(i,z)=\sum_{j\in\S, j\neq i} (j-i)\mathbf1_{\Gamma_{ij}}(z),\quad i\in \S, \,z\in \R_+.\]
		Then, the solution $(\La(t))$ to the SDE 
		\begin{equation}\label{c-3.1}
			\La(t)= \La(0)+\int_0^t\int_{\R_+} h(\La(s-), z)\mathcal{N}(\d s,\d z)
		\end{equation} is a continuous-time Markov chain on $\S$ with transition rate matrix $Q$; see more details in \cite[p. 546]{Sh24}. 
		
		As a consequence, we can rewrite $(X^{\xi,\ell}(t),\La^\ell(t))$ as a solution to the following SFDE
		\begin{equation*}
			\begin{split}
				\d \begin{pmatrix}
					X(t)\\ \La(t)
				\end{pmatrix} & = \begin{pmatrix} b(X_t,\La_t)\\ 0\end{pmatrix} \d t+ \begin{pmatrix} \sigma(X_t,\La_t) & 0\\ 0 &0\end{pmatrix} \begin{pmatrix} \d W(t)\\ \d W(t)\end{pmatrix}\\  &\quad +\int_{\R_+}\begin{pmatrix} 0\\ h(\La(s-),z)\end{pmatrix} \mathcal{N}(\d s, \d z).
			\end{split}
		\end{equation*}
		Then, the segment process $(X^{\xi,\ell}_t,\La^\ell_t)$ is a Markov process on the infinite dimension space $\C\times\D$ 
		due to \cite[Proposition 3.3]{RRG}. 
	\end{proof}
	Note that for the solution $(X^{\xi,\ell}(t),\La^\ell(t))$ to SFDE \eqref{a-1},  $(X^{\xi,\ell}_t,\La^\ell (t))$ does not have Markov property. Hence, we need to study the segment process $(X^{\xi,\ell}_t,\La^\ell_t)$ in $\C\times\D$ in order to apply the Markov property. In the following, let $P_t$ denote the semigroup associated with the Markov process $(X_t,\La_t)_{t\geq 0}  $ on $\C\times\D$. For any $\nu\in \pb(\C\times \D)$, $\nu P_t$ represents the distribution of $(X_t,\La_t)$ with random initial value $(X_0,\La_0)$  whose distribution is $\nu$. So, 
	\[\nu P_t(F)=\E[F(X_t,\La_t)]\quad \text{with $(X_0,\La_0)\sim \nu$},\quad F\in \B_b(\C\times\D).\]

	\section{Invariant probability measure:  finite state space case} \label{sec-fin}
	
	In this section, we study the existence and uniqueness of invariant probability measure for the process $(X_t,\La_t)$ satisfying \eqref{a-1} when the state space $\S=\{1,2,\ldots,N\}$ is a finite state space, i.e. $N<\infty$. The situation $N=\infty$ will be studied in the next section. 
	
	To study the invariant probability measure of the semigroup $P_t$ associated with the Markov process $(X_t,\La_t)$, it is known that the Feller property plays important role.  So, we first study the Feller property of  $P_t$ based on the regularity of the solution to \eqref{a-1}.

	Let $\theta:\S\to \R$ be a function, and put 
	\[\wt Q=Q+\mathrm{diag}(\theta_1,\theta_2,\ldots,\theta_N),\quad \eta_1(\wt Q)=-\max\{\mathrm{Re}(\gamma);\,\gamma\in \mathrm{Spec}(\wt Q)\},
	\]
	where $\mathrm{diag}(a_1,a_2,\ldots,a_N)$ denotes the diagonal matrix with diagonal $(a_1,a_2,\ldots,a_N)$, $\mathrm{Re}(\gamma)$ stands for the real part of $\gamma$ and $\mathrm{Spec}(\wt Q)$ denotes the spectrum of the operator $\wt Q$.
	\begin{mylem}[\cite{BGM}, Proposition 4.1]\label{lem-2.4}
		Let $(\La(t))$ be a continuous-time Markov chain on the finite state space $\S$ with transition rate matrix $Q=(q_{ij})$. Then, there are constants $c_1,c_2>0$ such that 
		\[c_1\e^{-\eta_1(\wt Q) t}\leq \E_i\Big[\e^{\int_0^t \theta_{\La(s)} \d s}\Big] \leq c_2\e^{-\eta_1(\wt Q) t},\quad t>0,\  i\in \S,\]
		where $\E_i[\,\cdot\,]$ denotes the expectation for   $(\La(t))$ with initial value $\La(0)=i$. 
	\end{mylem}
	\begin{myrem}\label{rem-2}
		Lemma \ref{lem-2.4} shows that the estimates of exponential functionals of Markov chains in terms of the quantity $\eta_1(\wt{Q})$ are quite sharp. Nevertheless, the argument of Lemma \ref{lem-2.4} essentially depends on the Perron-Frobenius theorem. Consequently, this lemma is valid only for finite-order matrices.  
	\end{myrem}
	
	Let   $\hat\alpha=\min_{i\in\S} \alpha_i$, $\hat{\beta}=\min_{i\in\S}\beta_i$, $\check{\alpha}=\max_{i\in\S} \alpha_i$ and $\check{\beta}=\max_{i\in\S} \beta_i$. 
	For $p\geq 0$, put
	\begin{equation}\label{c-2.5}
		Q_p=Q+p\,\mathrm{diag}\Big(\alpha_1\!+\!\beta_1\Big(\!\int_{-\tau_0}^0 \!\!\e^{\hat{\alpha}\theta} \mu(\d \theta)\Big), \ldots,\alpha_N\!+\! \beta_N \Big(\!\int_{-\tau_0}^0\!\!\e^{\hat{\alpha}\theta} \mu(\d \theta)\Big)  \Big),
	\end{equation}
	where $Q$  is the transition rate matrix of $(\La(t))$. 
	Define
	\begin{equation}\label{c-3}
		\begin{aligned}
			\eta_p&=-\sup\big\{\mathrm{Re}(\gamma);\ \gamma\in \mathrm{Spec}(Q_p)\big\},\quad p\geq 0  .
		\end{aligned}
	\end{equation}

	\begin{mylem}\label{lem-2.2}
		Assume $\mathrm{(H1)}$-$\mathrm{(H3)}$ hold. Let $(X^{\xi,\ell}(t),\La^\ell(t))$ be the solution to \eqref{a-1} with initial value $X_0^{\xi,\ell}=\xi\in \C$ and $\La_0^\ell=\ell\in \D$. Then there is a constant $c>0$, depending on $\tau_0,\check{\beta},\hat{\alpha}$, $\check{\alpha}$ and $\mu$,  such that   
		\begin{equation}\label{c-4}
			\E\big[\|X_t^{\xi,\ell}-X_t^{\eta,\ell}\|_\infty^2\big]\leq c\|\xi-\eta\|_\infty^2 \e^{ -\eta_1 t},
		\end{equation}
		\begin{equation}\label{c-5}
			\E\big[\|X_t^{\xi,\ell}\|_\infty^2\big]\leq c(1+\|\xi\|_\infty^2)\e^{-\eta_1t}
		\end{equation}
		for $t>0$, $\xi,\eta\in \C, \ell\in \D$.
	\end{mylem}
	
	\begin{proof}
		Let $\Gamma(t)=X^{\xi,\ell}(t)\!-\! X^{\eta,\ell}(t)$.  It\^o's formula and (H1) yield  that  
		\begin{align*}
			\d |\Gamma(t)|^2&\leq \Big(\alpha_{\La(t+\epsilon)}|\Gamma(t)|^2+\beta_{\La(t+\epsilon)} \int_{-\tau_0}^0\!|\Gamma(t+s)|^2 \mu(\d s)\Big)\d t+\d M(t),
		\end{align*}
		where $\d M(t)=2\big \la \Gamma(t), \big(\sigma(X_t^{\xi,\ell},\La_t^\ell)-\sigma( X_t^{\eta,\ell}, \La_t^{\ell})\big)\d W(t)\big\raa $. Using the independence between $(W(t))$ and $(\La(t))$ and the construction of probability space $(\Omega,\F,\p)$ in \eqref{b-0}, we take the conditional expectation $\E_{\p_1}$ to deduce that 
		\begin{equation}\label{c-5.0}
			\begin{split}
				&\E_{\p_1}\Big[ \e^{-\int_0^t \alpha_{\La(s+\epsilon)} \d s} |\Gamma(t)|^2\Big]\\  &\leq |\Gamma(0)|^2+\!\E_{\p_1}\Big[\int_0^t\!\! \beta_{\La(s+\epsilon)} \e^{-\int_0^s \alpha_{\La(r+\epsilon)} \d r} \int_{-\tau_0}^0 |\Gamma(s+r)|^2\mu(\d r)\d s\Big]\\
				&=|\Gamma(0)|^2\!+\!\int_{-\tau_0}^0\!\int_\theta^{t+\theta} \!\E_{\p_1}\!\big[\e^{-\int_0^{s } \alpha_{\La(r+\epsilon)} \d r}|\Gamma(s )|^2\big]\beta_{\La(s+\epsilon-\theta)} \e^{-\int_s^{s-\theta} \alpha_{\La(r+\epsilon)}\d r} \d s\mu(\d \theta)\\
				&\leq |\Gamma(0)|^2\!+\!\int_{-\tau_0}^0\!\int_{\theta}^0  \e^{-\int_0^{s-\theta} \alpha_{\La(r+\epsilon)}\d r} |\Gamma(s)|^2  \beta_{\La(s+\epsilon-\theta)} \d s \mu(\d \theta)\\
				&\quad +\int_0^t\! \E_{\p_1} \!\big[ \e^{-\int_0^s \alpha_{ \La(r+\epsilon)}\d r}  |\Gamma(s)|^2\big]\Big(\int_{-\tau_0}^0 \beta_{\La(s+\epsilon-\theta)} \e^{- \int_s^{ s-\theta}\! \alpha_{\La(r+\epsilon)} \d r} \mu(\d \theta)\Big)\d s\\
				&\leq |\Gamma(0)|^2\!+\!\int_{-\tau_0}^0\!\int_{\theta}^0 \check{\beta}\e^{\hat{\alpha}(s-\theta)} |\Gamma(s)|^2\d s \mu(\d \theta)\\
				&\quad +\int_0^t\! \E_{\p_1} \!\big[ \e^{-\int_0^s \alpha_{ \La(r+\epsilon)}\d r}  |\Gamma(s)|^2\big]\Big(\int_{-\tau_0}^0 \beta_{\La(s+\epsilon-\theta)} \e^{- \int_s^{ s-\theta}\! \alpha_{\La(r+\epsilon)} \d r} \mu(\d \theta)\Big)\d s\\
				&\leq   |\Gamma(0)|^2\!+\! \check{\beta} \e^{\hat{\alpha}\tau_0}\|\Gamma_0\|_\infty^2 \tau_0 \!+\!\int_0^t\!\Upsilon(s) \E_{\p_1}\big[\e^{-\int_0^s\alpha_{\La(r+\epsilon)} \d r}|\Gamma(s)|^2\big] \d s,
			\end{split}
		\end{equation}
		where 
		\begin{equation}\label{c-5.1}
			\Upsilon(s)=\int_{-\tau_0}^0\!\beta_{\La(s+\epsilon-\theta)} \e^{-\int_s^{s-\theta} \alpha_{\La(r+\epsilon)}\d r} \mu(\d \theta),\quad s\geq 0.
		\end{equation}
		By Gronwall's inequality, 
		\begin{equation}\label{c-5.2}
			\begin{split}
				\E_{\p_1}\big[ |\Gamma(t)|^2\big] 
				\leq \Big(|\Gamma(0)|^2\!+\! \check{\beta} \e^{\hat{\alpha}\tau_0}\|\Gamma_0\|_\infty^2 \tau_0\Big) \e^{ \int_0^t (\alpha_{\La(s+\epsilon)}+\Upsilon(s) )\d s} .
			\end{split}
		\end{equation}
		Since 
		\begin{align*}
			\int_0^t \Upsilon(s)\d s&=\int_{-\tau_0}^0\!\Big(\int_{ -\theta}^0+\int_{0}^t+\int_t^{t-\theta} \Big)\!\beta_{\La(s+\epsilon)} \e^{-\int_{s +\theta}^{s} \alpha_{\La(r+\epsilon)}\d r} \d s\mu(\d \theta) \\
			&\leq \int_{-\tau_0}^0\!\int_{-\theta}^0\!  \e^{\hat{\alpha} \theta} \beta_{\La(s+\epsilon)} \d s\mu(\d \theta)\!+\! \int_t^{t+\tau_0}\!\!\beta_{\La(s+\epsilon)}\Big( \int_{-\tau_0}^0 \!\!\e^{\hat{\alpha}\theta}\mu(\d \theta)\Big) \d s\\
			&\quad+\Big(\int_0^{t}\beta_{\La(s+\epsilon)} \d s\Big)\Big(\int_{-\tau_0}^0\!\e^{\hat{\alpha}\theta} \mu(\d \theta)\Big)\\
			&\leq \check{\beta}\tau_0\int_{-\tau_0}^0\!\e^{\hat{\alpha}\theta} \mu(\d \theta)+\Big(\int_0^{t}\beta_{\La(s+\epsilon)} \d s\Big)\Big(\int_{-\tau_0}^0\!\e^{\hat{\alpha}\theta} \mu(\d \theta)\Big)   
		\end{align*} for $t\geq 0$,
		we get from \eqref{c-5.2} that 
		\begin{equation}\label{c-5.3}
			\begin{split}
				&\E_{\p_1}\big[ |\Gamma(t)|^2\big]  
				\leq c   \|\Gamma_0\|_\infty^2
				\exp\Big(\int_0^t\!\Big(\!\alpha_{\La(s+\epsilon)}\!+\! \beta_{\La(s+\epsilon)}\!\int_{-\tau_0}^0\!\e^{\hat{\alpha}\theta} \mu(\d \theta)\Big)\d s  \Big),
			\end{split}
		\end{equation}
		where 
		\[c=\Big(1+\check{\beta} \e^{\hat{\alpha}\tau_0} \tau_0\Big) \e^{\check{\beta} \tau_0 \int_{-\tau_0}^0 \e^{\hat{\alpha} \theta}\mu(\d \theta)}.\]
		According to Lemma \ref{lem-2.4}, 
		\begin{equation}\label{c-5.4}
			\begin{split}
				&\E\big[|\Gamma(t)|^2\big] 
				\leq c \|\Gamma_0\|_\infty^2\e^{-\eta_1 t}, \quad t>0,
			\end{split}
		\end{equation} where $\eta_1$ is given in \eqref{c-3}.
		
		Next, for   $0\leq s<t$, using It\^o's formula, BDG's inequality and (H1), we have that 
		\begin{align*}
			\E\big[\sup_{s\leq r\leq t} |\Gamma(r)|^2\big] 
			&\leq \E|\Gamma(s)|^2 +c \int_{s-\tau_0}^t \E\big[|\Gamma(r)|^2\big] \d r \\
			&\quad +c\E\Big[\Big(\int_s^t\!|\Gamma(r)|^2 \|\sigma(X_r^{\xi,\ell},\La_r^\ell)-\sigma(X_r^{\eta,\ell},\La_r^\ell) \|_{\mathrm{HS}}^2\d r\Big)^{\frac 12}\Big] \\
			&\leq \E|\Gamma(s)|^2+c\int_{s-\tau_0}^t\!\E|\Gamma(r)|^2 \d r+\frac 12 \E\Big[\sup_{s\leq r\leq t} |\Gamma(r)|^2\Big].
		\end{align*}
		Hence, 
		\begin{equation}\label{c-5.5}
			\E\big[ \sup_{s\leq r\leq t}|\Gamma(r)|^2\big]
			\leq \E|\Gamma(s)|^2+c\int_{s-\tau_0}^t \E|\Gamma(r)|^2\d r.
		\end{equation}
		Invoking the estimate \eqref{c-5.4}, we finally get 
		\[\E\|\Gamma_t\|_\infty^2\leq c\|\xi-\eta\|_\infty^2 \e^{-\eta_1 t},\]
		which is the desired estimate \eqref{c-4}. 
		
		The estimate \eqref{c-5} can be deduced similarly as above by replacing $\Gamma(t)$ with $X^{\xi,\ell}(t)$ under the condition (H2), whose proof is omitted.  
	\end{proof}
	
	Based on the continuous dependence of $(X_t^{\xi,\ell})$ on the initial value $\xi\in\C$ established in Lemma \ref{lem-2.2}, we proceed to study the Feller property of the segment process $(X_t^{\xi,\ell},\La_t^{\ell})$.  
	
	\begin{mylem}\label{lem-2.3} Assume that $\mathrm{(A1)}$, $\mathrm{(A2)}$, $\mathrm{(H1)}$-$\mathrm{(H3)}$ hold and $\S$ is a finite state space. 
		The semigroup $P_t$ associated with the segment process $(X_t,\La_t)_{t\geq 0}$ admits Feller property. Namely, for any bounded continuous function $F$ on $\C\times \D$,    $(\xi_m,\ell_m),(\xi,\ell)\in \C\times\D$ satisfying $\lim_{m\to\infty} d\big( (\xi_m,\ell_m), (\xi, \ell)\big)=0$, it holds that  
		\begin{equation}\label{c-6}
			\lim_{m\to\infty} \delta_{(\xi_m,\ell_m)}P_t(F)=\delta_{(\xi,\ell)}P_t(F),
		\end{equation}
		where $\delta_{(\xi,\ell)}$ denotes the Dirac measure at the point $(\xi,\ell)$. 
	\end{mylem}
	
	\begin{proof} In view of Lemma \ref{lem-2.2}, now we mainly need to overcome the difficulty caused by the different initial values $\ell_m$ of the segment process $(X_t^{\xi_m,\ell_m},\La_t^{\ell_m})$. 
		The convergence of $\ell_m$ to $\ell$ in $\D$ as $m\to \infty$ implies that there exists a sequence of functions $\theta_m\in \Theta$ such that 
		\[\lim_{m\to \infty} \ell_m(\theta_m(t))=\ell(t),\]
		and $\lim_{m\to \infty }\theta_m(t)=t$ uniformly in $t\in [-\tau_0,0]$, where $\Theta$ is defined in \eqref{Theta}. In particular, this yields that 
		\[\lim_{m\to\infty} \ell_m(\theta_m(0))=\lim_{m\to\infty} \ell_m(0)=\ell(0),\]
		as $\theta(0)=0$ for any $\theta\in \Theta$ by the definition. As $\S$ is endowed with the discrete metric $\vartheta(i,j)$, it must hold that $\ell_m(0)=\ell(0)$ when $m\geq m_0$ for certain $m_0\in \N$. Without loss of generality, we assume $m_0=1$. Hence, for every $m\geq 1$, $\La^{\ell_m}(0)=\La^{\ell}(0)=\ell(0)$, and further by the Markov property, $\La^{\ell_m}(t)=\La^{\ell}(t)$ for all $t>0$, and $\La^{\ell_m}_t=\La^{\ell}_t$ for all $t\geq \tau_0$.  
		
		In view of $\La^{\ell_m}_t=\La^{\ell}_t$ for $t\geq \tau_0$, applying Lemma \ref{lem-2.2}, for $t\geq \tau_0$,
		\begin{equation}\label{c-6.1}
			\begin{split}
				&\E\big[ \E\big[\|X_t^{\xi_m,\ell_m}- X_t^{\xi,\ell}\|_\infty^2\big|\F_{\tau_0}\big] \big]\leq c\,\E\big[\|X_{\tau_0}^{\xi_m,\ell_m}-X_{\tau_0}^{\xi,\ell}\|_\infty^2\big]
				\e^{-\eta_1(t-\tau_0)}.
			\end{split}
		\end{equation}
		
		Now, we estimate $\E[\|X_t^{\xi_m,\ell_m}-X_t^{\xi,\ell}\|_\infty^2]$ for $t\in (0,\tau_0]$. For the sake of simplicity in notation, we put 
		\[X^m(t)=X^{\xi_m,\ell_m}(t),\ \La^m(t)=\La^{\ell_m}(t),\ X(t)=X^{\xi,\ell}(t),\ \La(t)=\La^{\ell}(t),\quad t\geq 0.\]
		Let $\Gamma(t)=X^m(t)-X(t)$, then it follows from It\^o's formula and  (H1), (H3) that
		{\small \begin{align*}
				\d |\Gamma(t)|^2\!& \leq 2\la \Gamma(t), b(X_t^m,\La_t^m)-b(X_t,\La_t^m)+b(X_t,\La_t^m)-b(X_t,\La_t)\raa \d t\\
				&\quad +\big( 2\|\sigma(X_t^m,\La_t^m) -\sigma(X_t,\La_t^m)\|_{\mathrm{HS}}^2 +2 \|\sigma(X_t,\La_t^m)-\sigma(X_t,\La_t)\|_{\mathrm{HS}}^2\big)\d t +d M(t)\\
				&\leq \Big[ (\alpha_{\La^m(t+\epsilon)}\!+\!\chi)|\Gamma(t)|^2\!+\! \kappa_3(1+\|X_t\|_\infty^2)\rho(\La_t^m,\La_t)^2
				\!\\
				&\qquad \qquad \qquad\quad +\! (\kappa_1\!+\!\beta_{\La^m(t+\epsilon)})\! \int_{-\tau_0}^0\!\! |\Gamma(t\!+\! s)|^2 \mu(\d s)\Big]\d t \! +\! \d M(t), 
		\end{align*} }%
		where $\d M(t)=2\la \Gamma(t), (\sigma(X_t^m,\La_t^m)-\sigma(X_t,\La_t))\d W(t)\raa$ and $\chi=1+\kappa_1$. Taking expectation w.r.t. $\E_{\p_1}$, we arrive at
		\begin{align*}
			\E_{\p_1}\Big[ &\e^{-\int_0^t(\alpha_{\La^m(s+\epsilon)}+\chi)\d s} |\Gamma(t)|^2\Big] 
			\leq |\Gamma(0)|^2\!+\! \kappa_3(1+\E_{\p_1}\big[\|X_t\|_\infty^2\big] )\!\int_0^t\rho(\La_s^m,\La_s)^2\d s \\
			& + \int_{-\tau_0}^0 \int_{\theta}^{t+\theta} \E\Big[ \e^{-\int_0^s (\alpha_{\La^m(r+\epsilon)}+\chi)\d r}|\Gamma(s)|^2\Big] \beta_{\La^m_{s+\epsilon-\theta}} \e^{-\int_s^{s-\theta} (\alpha_{\La^m(r +\epsilon)}+\chi)\d r} \d s\, \mu(\d\theta)\\
			&\leq |\Gamma(0)|^2\!+\! \kappa_3(1\!+\!\E_{\p_1} \big[\|X_t\|_\infty^2\big])\! \int_0^t\rho(\La_s^m,\La_s)^2\d s 
			+\|\Gamma_0\|_\infty^2\e^{-(\hat{\alpha}-\chi)\tau_0} \int_{-\tau_0}^0\!
			\wt{\Upsilon}(s)\d s\\
			&\quad+\int_0^t\wt{\Upsilon}(s)\E_{\p_1}\big[\e^{-\int_0^s
				\alpha_{\La^m(r+\epsilon)}\d r} |\Gamma(s)|^2\big]\d s,
		\end{align*}
		where 
		\[\wt{\Upsilon}(s) =\int_{-\tau_0}^0\!\!(\kappa_1\!+\! \beta_{\La^m(s+\epsilon-\theta)}) \e^{ -\int_{s}^{s-\theta}(\alpha_{\La^m(r+\epsilon)}+\chi)\d r} \mu(\d \theta).
		\]
		By Gronwall's inequality and taking expectation w.r.t.\! $\E$, also noticing that (H2) implies the boundedness of $\E\big[\|X_t\|_\infty^2\big]$ for $t\in [0,\tau_0]$ by Lemma \ref{lem-2.2},  we get that for some $c\geq1$
		\begin{align*}
			\E  \big[|\Gamma(t)|^2\big]  
			&\leq c\Big( \|\xi_m-\xi\|_\infty^2 \!+\! \int_0^{\tau_0} \E[\rho(\La_s^m,\La_s)^2]\d s   \Big) \e^{-\eta_1 t   +  t\int_{-\tau_0}^0   \kappa_1  \e^{(\hat{\alpha}-\chi)\theta}\mu(\d \theta)}.
		\end{align*}
		It is clear that the right-hand side of above inequality tends to $0$ as $m\to \infty$, and hence 
		\begin{equation*} 
			\lim_{m\to\infty} \E\big[|X_t^m-X_t|^2\big]=0,\quad t\in [0,\tau_0].
		\end{equation*}
		Using the same deduction as in the last part of the proof of Lemma \ref{lem-2.2}, we can get that 
		\begin{equation}\label{c-6.2}
			\lim_{m\to\infty} \E\big[\|X_t^m-X_t\|_\infty^2\big]=0,\quad t\in [0,\tau_0].
		\end{equation}
		Inserting \eqref{c-6.2} into \eqref{c-6.1}, we can show that for any $t>0$, 
		\begin{equation}\label{c-6.3}
			\lim_{m\to\infty} \E\big[\|X_t^m-X_t\|_\infty^2\big]=0.
		\end{equation}
		Consequently, for any bounded continuous function $F$ on $\C\times \D$, by the dominated convergence theorem, we derive from \eqref{c-6.3} that for any $t>0$,  
		\[\lim_{m\to\infty} \E F(X_t^{\xi_m,\ell_m}, \La_t^{\ell_m})=\E F(X_t^{\xi,\ell},\La_t^{\ell}),
		\]which is just the desired \eqref{c-6}. 
	\end{proof}
	
	Next, we shall use the coupling method to show the existence of the unique invariant measure for the segment process $(X_t,\La_t)$ and provide an explicit estimate of the convergence rate in the Wasserstein distance. 
	\begin{mythm}\label{main-1}
		Assume that $\mathrm{(A1)}$, $\mathrm{(A2)}$, $\mathrm{(H1)}$-$\mathrm{(H3)}$ hold and $\S$ is a finite state space. Suppose that $\eta_1 $ defined in \eqref{c-3} is positive. Then, there is some constant $c>0$ such that for any $\nu_1,\nu_2\in \pb_{1}(\C\times\D)$, for $t>(1+\eta_1/\lambda)\tau_0$, 
		\begin{equation}\label{d-1}
			\W_{1}(\nu_1P_t, \nu_2P_t) \leq c\Big( 1+\int_{\C} \|\xi\|_{\infty}\nu_1(\d \xi, \D)+\int_{\C} \|\eta\|_\infty\nu_2(\d \eta, \D)\Big)\e^{-\frac{\lambda \eta_1 t}{2(\lambda+\eta_1)}},
		\end{equation}   where $\lambda$ is a positive constant satisfying $\lambda<\|V\|_\infty^{-1}$, $V$ is given in $\mathrm{(A2)}$. Moreover, there exists a unique invariant probability measure $\nu_\infty\in \pb_1(\C\times \D)$ such that for $t>(1+\eta_1/\lambda)\tau_0$,
		\begin{equation}\label{d-2}
			\W_{1}(\delta_{(\xi,\ell)}P_t,\nu_\infty)\leq c\Big( 1+\|\xi\|_\infty+\int_{\C}\! \|\eta\|_\infty \nu_{\infty}(\d \eta, \D)\Big)\e^{-\frac{\lambda \eta_1 t}{2(\lambda+\eta_1)}} 
		\end{equation} for any $(\xi,\ell)\in \C\times\D$. 
	\end{mythm}
	
	\begin{proof}
		For any  given $(\xi,\ell), \,(\eta,\tilde \ell)\in \C\times\D$, let us consider the coupling process $(\La_t,\tilde \La_t)$ associated with the transition rate matrix $\wt Q$ given in (A2). Then, after the coupling time $T$ defined in \eqref{b-2}, $(\La(t))$ and $(\tilde \La(t))$ will move together. 
		Let $(X^{\xi, \ell}(t), \La(t))$ and $(X^{\eta,\tilde \ell}(t), \tilde \La(t))$ be the solution to \eqref{a-1} with the initial value $X^{\xi, \ell}_0=\xi$, $\La(0)=\ell(0)$, and $X^{\eta, \tilde \ell}_0=\eta$, $\tilde \La(0)=\tilde \ell(0)$ respectively.
		
		For $\beta\in (0,1)$ and $t> \frac{\tau_0}{1-\beta}$, by the definition of the Wasserstein distance
		\begin{align*}
			\W_1(\delta_{(\xi,\ell)}P_t,\delta_{(\eta,\tilde \ell)} P_t)
			&\leq \E\big[ \|X_t^{\xi,\ell}-X_t^{\eta,\tilde\ell}\|_{\infty}+\rho(\La_t^{\ell}, \La_t^{\tilde \ell})\big]\\
			&=\E\big[\big( \|X_t^{\xi,\ell}-X_t^{\eta,\tilde\ell}\|_{\infty}+\rho(\La_t^{\ell}, \La_t^{\tilde \ell})\big)\big( \mathbf1_{T\leq \beta t}+\mathbf1_{T>\beta t}\big)\big]
		\end{align*}
		As $t\!>\!\frac{\tau_0}{1-\beta}$, this yields that $\beta t<t -\!\tau_0$. So, it follows from $T\leq \beta t$ that $\La_t^\ell=\tilde \La^{\tilde \ell}_t$. 
		Therefore, by Lemma \ref{lem-2.2} and \eqref{b-4}, 
		\begin{align*}
			&\W_1(\delta_{(\xi,\ell)}P_t,\delta_{(\eta,\tilde \ell)} P_t)\\
			&\leq \E\big[ \E\big[ \|X_t^{\xi,\ell}-X_t^{\eta,\tilde\ell}\|_{\infty} \big|\F_T\big]\mathbf1_{T\leq \beta t}\big] + \E\big[ \big( \|X_t^{\xi,\ell}\|_\infty+\|X_t^{\eta,\tilde \ell}\|_\infty +1\big)\mathbf1_{T>\beta t}\big]\\
			&\leq \E\big[ \big(\E\big[ \|X_t^{\xi,\ell}-X_t^{\eta,\tilde\ell}\|_{\infty}^2 \big|\F_T\big]\big)^{\frac 12}\mathbf1_{T\leq \beta t}\big] + \E\big[ \big( \|X_t^{\xi,\ell}\|_\infty+\|X_t^{\eta,\tilde \ell}\|_\infty +1\big)\mathbf1_{T>\beta t}\big]\\
			&\leq c^{\frac 12} \E\big[\|X_T^{\xi,\ell}-X_T^{\eta,\tilde \ell}\|_\infty \big]\e^{-\frac 12 (1-\beta)\eta_1 t} + \E\big[(\|X_t^{\xi,\ell}\|_\infty+\|X_t^{\eta,\tilde \ell}\|_\infty+1)^2\big]^{\frac 12} \p(T>\beta t)^{\frac 12}\\
			&\leq C(1+\|\xi\|_\infty+\|\eta\|_\infty)\Big(\e^{-\frac 12 (1-\beta)\eta_1 t}+\frac{\e^{-\frac12 \lambda \beta t}}{(1-\lambda\|V\|_\infty)^{1/2}}\Big).
		\end{align*}
		By taking $\beta=\frac{\eta_1}{\lambda+\eta_1}\in (0,1)$, the previous inequality yields that 
		\begin{equation}\label{d-3}
			\W_1(\delta_{(\xi,\ell)}P_t,\delta_{(\eta,\tilde \ell)} P_t)
			\leq c (1+\|\xi\|_\infty+\|\eta\|_\infty) \e^{-\frac{\lambda \eta_1 t}{2(\lambda+\eta_1)}}
		\end{equation} for $t>\frac{(\lambda+\eta_1)\tau_0}{\lambda}$ and some constant $c>0$ depending on $\|V\|_\infty$ and $\lambda$. 
		As a consequence, for each $\pi\in \C(\nu_1,\nu_2)$, 
		\begin{align*}
			&\W_1(\nu_1P_t,\nu_2P_t)\leq \int_{(\C\times\D)^2}\W_1(\delta_{(\xi,\ell)}P_t,\delta_{(\eta,\tilde \ell)} P_t) \pi\big((\d \xi,\d \ell), (\d \eta,\d \tilde \ell)\big)\\
			& \leq c\Big(1+\!\int_{\C}\!\|\xi\|_\infty\nu_1(\d \xi,\D)\!+\!\int_{\C} \|\eta\|_\infty\nu_2(\d \eta,\D)\Big) \e^{-\frac{\lambda \eta_1 t}{2(\lambda +\eta_1)}},\quad t>\frac{(\lambda\!+\!\eta_1)\tau_0}{\lambda},
		\end{align*}
		which is the desired estimate \eqref{d-1}. 
		
		For any given $\nu\in \pb_1(\C\times\D)$ and a fixed $t_0>(1+\eta_1/\lambda)\tau_0$, by the semigroup property of $P_t$, we derive from \eqref{d-1} that 
		\begin{align*}
			\W_1\big(\nu P_{nt_0}, \nu P_{(n+1)t_0}\big)&=\W_1 \big (\nu P_{nt_0}, (\nu P_{t_0}) P_{nt_0}\big)\\
			&\leq c\Big( 1+\!\int_{\C}\!\|\xi\|_\infty \nu(\d \xi, \D)\!+\!\int_{\C}\!\! \|\xi\|_\infty (\nu P_{t_0})(\d \xi, \D)\Big) \e^{-\frac{\lambda \eta_1 nt_0 }{2(\lambda+\eta_1)}},
		\end{align*} 
		whence, $(\nu P_{nt_0})_{n\geq 1}$ is a Cauchy sequence in the Polish space $(\pb_1(\C\times\D), \W_1)$, so there is a probability measure $\nu_\infty \in \pb_1(\C\times \D)$ such that 
		\begin{equation}\label{d-4}
			\lim_{n\to\infty} \W_1(\nu P_{nt_0},\nu_\infty)=0.
		\end{equation}
		
		For any $\tilde \nu\in \pb_1(\C\times\D)$ and any $F\in \mathrm{Lip}_1(\C\times \D)$, 
		\begin{align*}
			&\big|\tilde \nu P_{nt_0} (F)-\nu_\infty (F)\big|\\
			&\leq \int_{\C\times\D}\!\! \big|\delta_{(\xi,\ell)} P_{nt_0}(F)\!-\!\nu_\infty(F)\big| \tilde \nu(\d \xi, \d \ell)\\
			&\leq \int_{\C\times\D}\!\! \big|\delta_{(\xi,\ell)}P_{nt_0}(F)\!-\! \nu P_{nt_0}(F)\big|\tilde \nu(\d \xi,\d \ell) \!+\! \big|\nu P_{nt_0}(F)\!- \! \nu_\infty(F)\big|\\
			&\leq \int_{\C\times\D}\! \int_{\C\times\D}\!\!\big|\delta_{(\xi,\ell)} P_{nt_0}(F) \!-\! \delta_{(\eta,\tilde \ell)}P_{nt_0}(F)\big| \nu(\d \eta,\d \tilde \ell) \tilde \nu(\d \xi, \d \ell) \!+\! \big|\nu P_{nt_0}(F)\!-\!\nu_\infty(F)\big|\\
			&\leq \int_{\C\times\D}\! \int_{\C\times\D}\!\! \W_1(\delta_{(\xi,\ell)} P_{nt_0}, \delta_{(\eta, \tilde \ell)}P_{nt_0})\nu(\d \eta,\d \tilde \ell) \tilde \nu(\d \xi, \d \ell)\!+\! \W_1(\nu P_{nt_0},\nu_\infty)\\
			&\leq c\Big(1+\!\int_{\C}\!\|\xi\|_\infty \tilde \nu(\d \xi, \D)\!+\!\int_{\C}\!\|\eta\|_\infty\nu(\d \eta,\D)\Big) \e^{-\frac{\lambda \eta_1 nt_0}{2(\lambda+\eta_1)}} \!+\! \W_1(\nu P_{nt_0},\nu_\infty)
		\end{align*}
		where we used the dual representation formula \eqref{b-7} of $\W_1$ in the second to last step, and  \eqref{d-3} in the last step. Using \eqref{b-7} again and the arbitrariness of $F\in \mathrm{Lip}_1(\C\times\D)$, we obtain that 
		\[\W_1(\tilde\nu P_{nt_0}, \nu_\infty)\leq c\Big(1+\!\int_{\C}\!\|\xi\|_\infty \tilde \nu(\d \xi, \D)\!+\!\int_{\C}\!\|\eta\|_\infty\nu(\d \eta,\D)\Big) \e^{-\frac{\lambda \eta_1 nt_0}{2(\lambda+\eta_1)}} \!+\! \W_1(\nu P_{nt_0},\nu_\infty),\]
		and by \eqref{d-4}
		\begin{equation}\label{d-5}
			\lim_{n\to\infty} \W_1(\tilde \nu P_{nt_0},\nu_\infty)=0.
		\end{equation}
		This implies that $\tilde \nu P_{nt_0}$ converges weakly to $\nu_\infty$ as $n\to\infty$(cf. \cite[Theorem 6.9]{Vil}). 
		
		For any $s>0$, any $(\xi,\ell)\in \C\times \D$ and any  bounded continuous function $h$ on $\C\times \D$,  the semigroup property of $P_t$ yields that 
		\begin{equation}\label{d-6}
			\delta_{(\xi,\ell)}P_{nt_0}P_s (h)=\delta_{(\xi,\ell)} P_{nt_0+s}(h)=\delta_{(\xi,\ell)}P_s P_{nt_0}(h).
		\end{equation}
		On the one hand, by Lemma \ref{lem-2.3}, the Feller property of $P_t$ yields that $\tilde h(\eta,\tilde \ell)=\delta_{(\eta, \tilde \ell)}P_s(h)$ is also a bounded continuous function on $\C\times\D$. So, \eqref{d-5} yields that 
		\begin{equation}\label{d-7}\lim_{n\to\infty} \delta_{(\xi,\ell)}P_{nt_0}P_s(h)=\lim_{n\to\infty} \delta_{(\xi,\ell)}P_{nt_0}(\tilde h) =\nu_\infty(\tilde h)=\nu_\infty(P_s(h)).
		\end{equation}
		On the other hand, $\delta_{(\xi,\ell)}P_s \in \pb_1(\C\times\D)$, so taking $\tilde \nu =\delta_{(\xi,\ell)}P_s$ in \eqref{d-5}, we get
		\begin{equation}\label{d-8}
			\lim_{n\to\infty} \delta_{(\xi,\ell)}P_s P_{nt_0}(h)=\nu_\infty(h).
		\end{equation}
		Inserting \eqref{d-7} and \eqref{d-8} into \eqref{d-6}, we obtain that 
		\[\nu_\infty P_s(h)=\nu_\infty (h),\]
		which means that $\nu_\infty \in \pb_1(\C\times \D)$ is an invariant probability measure of the semigroup $P_t$.  By taking $\nu_2=\nu_\infty$ in \eqref{d-1}, we can derive \eqref{d-2}. 
		
		At last, let $\hat{\nu}\!\in\! \pb_1(\C\!\times\!\D)$ be another invariant probability measure of $P_t$. By \eqref{d-2},
		\begin{align*}
			\W_1(\hat{\nu},\nu_\infty)&\leq c\Big(1+\int_{\C}\!\|\xi\|_\infty \hat{\nu}(\d \xi, \D)+\int_{\C} \|\xi\|_\infty \nu_\infty(\d \xi, \D)\Big)\e^{-\frac{\lambda \eta_1 t}{2(\lambda+\eta_1)}}\\
			&\longrightarrow 0,\qquad \text{as}\ t\to \infty.
		\end{align*}
		Therefore, $\hat{\nu}=\nu_\infty$, and we complete the proof.   
	\end{proof}
	
	\begin{myrem}\label{rem-4}
		In Theorem \ref{main-1}, we provides an estimate of convergence rate in $L_1$-Wasserstein distance in terms of $\eta_1$, which is related to the $Q$ matrix of $(\La(t))$ and constants $(\alpha_i),(\beta_i)$ given in (H1). The constants $(\alpha_i)$ and $(\beta_i)$ are closely related to the coefficients of SFDE \eqref{a-1}. Similar to $\eta_1$, similar criteria have been used in the study of Ornstein-Uhlenbeck diffusion process with regime-switching in \cite[Theorem 1.8]{BGM}. 
		
	\end{myrem}

	As an application of Theorem \ref{main-1}, let us provide an example.
	
	\begin{myexam}\label{exam-1}
		Let $(\La(t))$ be a Markov chain on $\S=\{1,2\}$ with transition rate matrix 
		$ Q=\begin{pmatrix} -1 &1 \\ 1 &-1\end{pmatrix}$. Consider the following SFDE on $\R$:
		\begin{equation}\label{ex-1}
			\d X(t)=\big(a_{\La(t-1)} X(t)+b_{\La(t-1)} X(t-1)\big)\d t+\sigma_{\La(t-1)} \d W(t),
		\end{equation} where $(W(t))$ is a 1-dimensional Brownian motion. In the following, we shall verify all the conditions of Theorem \ref{main-1}. (H2) holds obviously. 
		
		For (H1), using the notations of Theorem \ref{main-1}, we have
		\begin{align*}
			&2\la \xi(0)-\eta(0),b(\xi,\ell)-b(\eta,\ell)\raa +\|\sigma(\xi,\ell)-\sigma(\eta,\ell)\|_{\mathrm{HS}}^2\\
			&=2 a_{\ell(-1)}  |\xi(0)-\eta(0)|^2+2b_{\ell(-1)}\la \xi(0)-\eta(0), \xi(-1)-\eta(-1)\raa\\
			&\leq (2a_{\ell(-1)}+b_{\ell(-1)})|\xi(0)-\eta(0)|^2+b_{\ell(-1)}|\xi (-1)-\eta(-1)|^2.
		\end{align*} So we can take $\tau_0=1$, $\epsilon=-1$, $\mu(\d \theta)=\delta_{-1}(\d \theta)$, $\alpha_i\!=2a_i\!+\!b_i$ and $\beta_i\!=\!\max\{b_i,0\}$ for $i\in\!\S$. 
		
		For (H3), 
		\begin{align*}
			|b(\xi,\ell)-b(\xi,\tilde\ell)|&=|(a_{\ell(-1)}-a_{\tilde\ell(-1)})\xi(0) +(b_{\ell(-1)}-b_{\tilde \ell(-1)})\xi(-1)|\\
			&\leq \big(|a_{\ell(-1)}-a_{\tilde \ell(-1)}|+|b_{\ell(-1)}-b_{\tilde\ell (-1)}|\big)\|\xi\|_\infty\\
			&\leq \big(2\max_{i}|a_i|+\max_{i}|b_i|\big)\|\xi\|_\infty\vartheta(\ell(-1),\tilde \ell(-1))\\
			&\leq \big(2\max_{i}|a_i|+\max_{i}|b_i|\big)\|\xi\|_\infty\rho(\ell,\tilde \ell).
		\end{align*}
		For (A2), we take $\wt{Q}$ to be the classical coupling of $Q$, and $V:\S\times\S\to[1,\infty)$ as
		\[V(1,1)=V(2,2)=1,\quad V(1,2)=V(2,1)=3/2.\]
		Then, by   Example \ref{coup-2}, we have $\tilde q(1,1;1,2)=\tilde q(1,1;2,1)=\tilde q(2,2;1,2)=\tilde q(2,2;2,1)=0$, $\tilde q(1,1;2,2)=\tilde q(2,2;1,1)=\tilde q(1,2;1,1)$ $=\tilde q(1,2;2,2)=\tilde q(2,1;1,1)=\tilde q(2,1;2,2)=1$, $\tilde q(1,1;1,1)= \tilde q(2,2;2,2)=-1$, $\tilde q(1,2;1,2)=\tilde q(2,1;2,1)=-2$. One can check directly that $\wt{Q}V(i,j)= -1$ for $i\neq j$.  For instance, $\wt{Q}V(1,2)=\tilde q(1,2;1,1)(V(1,1)-V(1,2))+\tilde q(1,2; 2,2)(V(2,2)-V(1,2))=-1$. 
		
		At last, we proceed to verify the condition $\eta_1>0$. 
		Calculating the determinant of the matrix $\lambda I_2-Q_1$ directly  by \eqref{c-2.5},
		\[|\lambda I_2-Q_1|=\lambda^2-(\zeta_1+\zeta_2-2)\lambda-\zeta_1- \zeta_2+\zeta_1\zeta_2  ,\] 
		where $I_2$ denotes $2\times 2$ identity matrix, $\zeta_1= \alpha_1+\beta_1\e^{-\alpha_1}$, $\zeta_2= \alpha_2+\beta_2\e^{-\alpha_1}$ if $\alpha_1<\alpha_2$. Then, $|\lambda I_2-Q_1|=0$ has two negative real roots if $\zeta_1+\zeta_2<\min\{\zeta_1\zeta_2, 2\}$, which yields immediately that $\eta_1>0$.
		
		In all, put  $\alpha_i=2a_i+b_i$, $\beta_i=\max\{b_i,0\}$ for $i=1,2$. If $\alpha_1<\alpha_2$,  put $\zeta_1=\alpha_1+\beta_1\e^{-\alpha_1}$, $\zeta_2=\alpha_2+\beta_2\e^{-\alpha_1}$. If  
		\begin{align*}
			&\alpha_1<\alpha_2, \ \ \zeta_1+\zeta_2<\min\{\zeta_1\zeta_2,2\},
		\end{align*}
		then all the conditions of Theorem \ref{main-1} are satisfied. 
		
		More precisely,  when taking $\alpha_1=-4$, $\alpha_2=\frac1{20}$, $\beta_1=\beta_2=\frac 12\e^{-4}$, $\sigma_1=\sigma_2=1$, we have $\eta_1>0$. Then, the segment process $(X_t,\La_t)$ admits a unique invariant probability measure $\nu_\infty$ on $\C\times\D$.  
	\end{myexam}
	
	\section{Invariant probability measure: infinite state space case}\label{sec-inf}
	
	The section is devoted to the study of invariant probability measures for the solutions to functional SDE \eqref{a-1} with regime-switching on an infinite state space $\S=\{1,2,\ldots,N\}$ with $N=\infty$. 
	
	From the investigation of Section \ref{sec-fin} in the case $N<\infty$, we see that the estimation of the expectation of the exponential functionals of the Markov chain $(\La(t))$ in Lemma \ref{lem-2.4} plays the crucial role. However, it is valid only for matrices of finite order; see Remark \ref{rem-2}. 
	
	To overcome this difficulty, we shall develop a truncation method based on a construction of coupling processes in terms of a Poisson random measure. The idea behind the construction of coupling processes has been employed in \cite{Sh24} based on Skorokhod's representation of Markov chain (cf. \cite{Sh15c,YZ10}). By utilizing a common Poisson random process, we construct two families of intervals on the real line, where the lengths of the intervals correspond to the off-diagonal elements o the two transition rate matrices. By controlling the arrangement order of these intervals, we ultimately construct a coupling process that preserves the stochastic ordering of the two Markov chains pathwisely. 
	
	
	For an $M\in \N$, let 
	\[\S_M=\{1,2,\ldots,M\}.\]
	We shall construct a continuous-time Markov chain $(\hat{\La}(t))_{t\geq 0}$ on $\S_M$ such that 
	\[\p(\hat{\La}(t)\leq \La(t),\  \forall\,t\geq 0)=1.\]
	Assume in addition that 
	\begin{itemize}
		\item[(A3)] For each $i\in \S$, there is a $c_i\in \N$ such that $q_{ij}=0$ for all $j\in \S$ with $j>i+c_i$.
	\end{itemize}

	\begin{myprop}[Comparison coupling]\label{prop-3.1}
		Suppose $\mathrm{(A1)}$-$\mathrm{(A3)}$ hold. Define a transition rate matrix $\hat{Q}=(\hat{q}_{ij})_{i,j\in \S_M} $ on $\S_M$ by 
		\begin{equation}\label{e-1}
			\hat{q}_{ij}=\begin{cases}
				\sup_{\ell\in \S,\ell\geq i}q_{\ell j}, & \text{if $j<i$},\\ \min_{i\leq \ell<j}q_{\ell j}, &\text{if $j>i$}, 
			\end{cases}\quad \hat{q}_{ii}=-\sum_{k\neq i} \hat{q}_{ik},\quad \forall\,i,j\in \S_M.
		\end{equation} Then for any $i_0\in \S$ and ${i}'_0\in \S_M$ satisfying $ {i}_0'\leq i_0$, there is a coupling process $(\La(t),\hat{\La}(t))_{t\geq 0}$ with initial value $(\La(0),\hat{\La}(0))=(i_0,{i}'_0)\in \S\times\S_M$ such that $(\La(t))$ is a Markov chain on $\S$ with transition rate matrix $Q=(q_{ij})_{i,j\in \S}$ and $(\hat{\La}(t))$ is a Markov chain on $\S_M$ with transition rate matrix $\hat{Q}$
		satisfying  
		\begin{equation}\label{e-2}
			\p\big(\hat{\La}(t)\leq \La(t),\ \forall \,t\geq 0\big)=1.
		\end{equation}
	\end{myprop}
	
	The proof of Proposition \ref{prop-3.1} follows the methodology outlined in   \cite{Sh24}, with two key distinctions that must be highlighted:
	\begin{enumerate}
		\item \textbf{Recognition of fundamental differences between upper and lower Bounds} In the previous work \cite{Sh24},  an upper bound  was successfully constructed for the random process $(\Lambda_t)$. However, at that time, we did not fully recognize the fundamental differences between the construction of upper bounds and lower bounds. We inadvertently applied a similar strategy for the lower bound, i.e. $q_{ij}^\ast =\sup_x\max_{ j<\ell\leq i}q_{\ell j}(x)$ when $j<i$ similar to the construction of upper bound. In the current paper, we have identified this important distinction. This recognition led us to provide a rigorous proof for the lower bound.
		
		\item \textbf{Simplification and removal of state space constraints} A key insight gained from re-examining the proof techniques was that the restriction on the state space (i.e. $(\Lambda_t)$ and $(\Lambda^\ast_t)$ have the same state space and started from the same point) imposed in [20] was not essential. By simplifying the proof structure and focusing on the intrinsic properties of the process, we were able to remove this state space constraint. This simplification, while  technically manageable, is  theoretically significant because it broadens the applicability of our result to a generalization of the Perron-Frobenius theorem.  The classical Perron-Frobenius theorem plays important role in many researches, but it works only for matrices of finite-order.  By projecting the infinite countable state space onto finite dimensions, we are able to leverage the well-established Perron-Frobenius theory for finite-order matrices to estimate quantities associated with Markov chains on infinite state spaces, such as exponential functionals. Moreover, our finite projection method preserves the order relation in the sense of trajectories.
	\end{enumerate}
	
	\begin{proof}
		The proof is divided into four steps.
		
		\noindent\textbf{Step 1}.(\emph{Construction of Poisson random measure}) Denote by $\bm{m}(\d x)$ the Lebesgue measure over $\R$. Let $\xi_{k}$, $k\geq 1$, be random variables taking values in $[-K_0,K_0]$ with 
		\[\p(\xi_k\in \d x)=\bm{m}(\d x)/2K_0,\]
		and $\tau_k$, $k\geq 1$, be nonnegative random variables such that $\p(\tau_k>t)=\e^{-2K_0 t}$, $t\ge 0$. $\{\xi_k\}_{k\geq 1}$ and $\{\tau_k\}_{k\geq 1}$ are mutually independent. 
		Let 
		\[\zeta_0=0,\ \zeta_n=\tau_1+\tau_2+\cdots+\tau_n,\  n\geq 1,\]
		and denote 
		\[\D_{\bm{p}}=\cup_{n\geq 1}\{\,\zeta_n\,\},\quad \text{and}\quad \bm{p}(t)=\sum_{0\leq s<t}\Delta \bm{p}(s),\]
		where $\Delta \bm{p}(s)=\bm{p}(s)-\bm{p}(s-)$ for $s\not\in\D_{\bm{p}}$, and $\Delta \bm{p}(\zeta_n)=\xi_n$ for $n\geq 1$. The finiteness of $K_0$ yields that $\lim_{n\to\infty}\zeta_n=\infty$ almost surely. 
		Let 
		\begin{equation}\label{e-3}
			\mathcal{N}_{\bm{p}}([0,t]\!\times\! A)=\#\{s\in \D_{\bm{p}};\ 0\leq s\leq t,\,\Delta\bm{p}(s)\in A\}.
		\end{equation}
		Consequently, $(\bm{p}(t))$ and $\mathcal{N}_{\bm{p}}(\d t,\d z)$ are respectively a Poisson point process and a Poisson random measure with intensity measure $\d t\bm{m}(\d z)$(cf. \cite[Chapter I]{IW}). 
		
		\noindent \textbf{Step 2}.(\emph{Construction of intervals}) We shall use the method of \cite{Sh24} to construct the desired intervals in order to apply Skorokhod's representation theorem. The key point is to construct the intervals $\Gamma_{ij}$ for $j<i$ on $[0,\infty)$, but the intervals $\Gamma_{ij}$ for $j>i$ on $(-\infty, 0]$. 
		
		For $i\in \S$, let 
		\begin{align*}
			&\Gamma_{i1}=[0,q_{i1}),\quad 
			\Gamma_{i2}=[q_{i1},q_{i1}+q_{i2}),\quad \ldots, \ \Gamma_{ii-1}=\Big[\sum_{j=1}^{i-2}q_{ij},\sum_{j=1}^{i-1} q_{ij}\Big).
		\end{align*}
		For convenience of notation, let $\Gamma_{ii}=\emptyset$.  Due to (A3), put $\Gamma_{ij}=\emptyset$ for $j-i>c_i$. Let 
		\begin{align*}
			&\Gamma_{ii+c_i}=[-q_{ii+c_i},0),\qquad  
			\Gamma_{ii+c_i-1} =[-q_{ii+c_i-1}-q_{ii+c_i},-q_{ii+c_i}),\quad \ldots,\\
			&\Gamma_{ii+1}=\Big[-\sum_{j=i+1}^{i+c_i} q_{ij},-\sum_{j=i+2}^{i+c_i} q_{ij}\Big).
		\end{align*}
		Now we construct the intervals associated with $( \hat{q}_{ij} )_{i,j\in\S_M}$ in the same way by noting that the state space is now $\S_M=\{1,2,\ldots,M\}$. For $i\in \S_M$, let 
		\begin{align*}
			\hat{\Gamma}_{i1} &= [0,\hat{q}_{i1}),\quad \hat{\Gamma}_{i2} =[\hat{q}_{i1} , \hat{q}_{i1} + \hat{q}_{i2}),\ldots, \quad 
			\hat{\Gamma}_{ii-1}  =\Big[\sum_{j=1}^{i-2}\hat{q}_{ij},\sum_{j=1}^{i-1} \hat{q}_{ij} ),
		\end{align*}
		and 
		\begin{align*}
			\hat{\Gamma}_{iN} &=[-\hat{q}_{iN},0),\quad \hat{\Gamma}_{iN-1}=[-\hat{q}_{iN-1}-\hat{q}_{iN},- \hat{q}_{iN}),\ldots,\\
			\hat{\Gamma}_{ii+1}&=\Big[-\sum_{j=i+1}^N \hat{q}_{ij} ,-\sum_{j=i+2}^N \hat{q}_{ij} \Big).
		\end{align*}
		According to the construction of these intervals, we deduce from the fact that $\hat{q}_{ir}  \leq q_{kr} $ for $i\leq k< r$,  $\hat{q}_{ir}  \geq q_{kr} $ for $r< i\leq k$  that for any  $i\in \S_M$, $k\in \S$ with $i\leq k$,
		\begin{align}\label{e-4}
			&\bigcup_{M\geq r\geq m}\hat{\Gamma}_{ir} \subseteq \bigcup_{r\geq m}\Gamma_{kr},\quad \ \ \forall\,m\geq k,\\ \label{e-5}
			&\bigcup_{1\leq r\leq m}\hat{\Gamma}_{ir} \ \supseteq \bigcup_{1\leq r\leq m} \Gamma_{kr},\quad \forall\,m\leq i.
		\end{align}
		
		\noindent\textbf{Step 3}.(\emph{Construction of coupling processes}) Let 
		\begin{align*}
			\vartheta(i,z)&=\sum_{j\in \S,j\neq i} (j-i)\mathbf1_{\Gamma_{ij}}(z),\quad \text{for $i\in \S$},\\
			\hat{\vartheta} (i,z)&=\sum_{j\in \S_M,j\neq i}(j-i)\mathbf1_{\hat{\Gamma}_{ij}}(z),\quad \text{for $i\in \S_M$}.
		\end{align*}
		Let 
		\begin{align}\label{e-6}
			\La(t)&=\La(0)+\int_0^t \int_{\R}\!\vartheta(\La(t-),z) \mathcal{N}_{\bm{p}}(\d s,\d z),\\ \label{e-7}
			\hat{\La}(t)&=\hat{\La}(0)+\int_0^t \int_{\R}\!\hat{\vartheta} (\hat{\La}(t-),z)\mathcal{N}_{ \bm{p}}(\d s,\d z),
		\end{align}
		and $\La(0)\geq\hat{\La}(0) \in \S_M$. 
		Then, according to Skorokhod's representation theorem (cf. \cite{YZ10}, \cite[Sec.2.1]{Sh24}), $(\La(t))$ and $(\hat{\La}(t))$ are continuous-time Markov chains associated with $(q_{ij})_{i,j\in \S}$ and $(\hat{q}_{ij} )_{i,j\in \S_M}$ respectively. 
		
		\noindent\textbf{Step 4}.(\emph{Comparison in the pathwise sense})
		Due to \eqref{e-6} and \eqref{e-7}, $(\La(t))$ and $(\hat{\La} (t))$ have no jumps outside of $\D_{\bm{p}}$, so it is sufficient to show that 
		\[\p(\hat{\La} (t)\leq \La(t),\ t\in \D_{\bm{p}})=1.\]
		Assume that $i= \hat{\La}(\zeta_{n-1})\leq \La(\zeta_{n-1})=k$ for some $n\geq 1$, we go to prove that 
		\[\hat{\La} (\zeta_n)\leq \La(\zeta_n),\]
		whose proof is divided into three cases.
		\begin{itemize}
			\item[(i)] If for some $j\in \S_M$, $j\geq k$, $\hat{\La}  (\zeta_n)=j$, which implies that $k\in \S_M$ too. Then, by \eqref{e-7}, we have $\xi_n\in \hat{\Gamma}_{ij} \subset (-\infty, 0]$, and it follows from \eqref{e-4} that $\xi_n\in \bigcup_{r\geq j}\Gamma_{kr}$. Hence, by \eqref{e-6}, we have $\La(\zeta_n)\geq j=\hat{\La} (\zeta_n)$. 
			\item[(ii)] If for some $m\in \S_M$ with $i<m<k$, $\hat{\La} (\zeta_n)=m$. Then, $\xi_n\in \hat{\Gamma}_{im} $ thanks to \eqref{e-7}, and hence $\xi_n<0$. It is clear that $\xi_n\not\in \bigcup_{j\leq k}\Gamma_{kj}$ as $\bigcup_{j\leq k} \Gamma_{kj}\subset [0,\infty)$. So, 
			\[\La(\zeta_n)>k>\hat{\La} (\zeta_n).\]
			\item[(iii)] If $\hat{\La}(\zeta_n)=j\leq i$, we only need to study the situation that for some $m'<i$, $\La(\zeta_n)=m'$. Otherwise, it holds immediately that $\hat{\La}(\zeta_n)\leq \La(\zeta_n)$.  When $\La(\zeta_n)=m'$, then $\xi_n\in \Gamma_{km'}$ by \eqref{e-6}. Due to \eqref{e-5}, $\Gamma_{km'}\subset \bigcup_{r\leq m'} \hat{\Gamma}_{ir} $, and it follows from \eqref{e-7} that $\hat{\La} (\zeta_n)\leq m'$. Therefore, $\hat{\La} (\zeta_n)\leq \La(\zeta_n)$. 
			
		\end{itemize}
		In all, by induction on $n$  we can show that 
		\[\p(\hat{\La} (t)\leq \La(t),\ \forall\,t\geq 0)=1,\]
		and the proof is complete. 
	\end{proof}
	
	\begin{mylem}\label{lem-3.1}
		Assume that $\mathrm{(A1)}$-$\mathrm{(A3)}$ hold. Let $(\La(t))$ and $(\hat{\La}(t))$ be the coupling process constructed in Proposition \ref{prop-3.1} with $\S\ni\La(0)\geq\hat{\La}(0)\in \S_M$ for some $M\in \N$. Suppose that $\hat{Q}$ is irreducible. Let $\theta:\S\to\R$ be a decreasing function. Let
		\begin{equation}\label{e-7.5}
			\begin{aligned}
				\wt{Q}_M& =\hat{Q}+ \mathrm{diag}(\theta_1,\theta_2,\ldots,\theta_M),\\
				\eta_1(\wt{Q}_M)&=-\max\big\{ \mathrm{Re}(\gamma);\,\gamma\in \mathrm{Spec}(\wt{Q}_M)\big\}.
			\end{aligned}
		\end{equation}
		Then there exists a constant $c_3>0$ such that 
		\begin{equation}\label{e-8}
			\E\Big[\e^{\int_0^t\theta_{\La(s)}\d s}\Big]\leq \E\Big[\e^{\int_0^t \theta_{\hat{\La}(s)}\d s}\Big] \leq c_3 \e^{-\eta_1(\wt{Q}_M)t},\quad t\geq 0.
		\end{equation}
	\end{mylem} 
	
	\begin{proof}
		By Proposition \ref{prop-3.1}, 
		\[\hat{\La}(t)\leq \La(t),\quad t\geq 0,\quad \text{a.s.}.\]
		By the decreasing property of the function $\theta(\cdot)$, it holds
		\[\theta_{\La(t)} \leq \theta_{\hat{\La}(t)},\quad t\geq 0,\ \text{a.s.}. \]
		Applying Lemma \ref{lem-2.4} to the Markov chain $(\hat{\La}(t))$ on the finite state space $\S_M$, there is a constant $c_3>0$ such that \eqref{e-8} holds, and hence the proof is complete. 
	\end{proof}
	
	\begin{mylem}\label{lem-3.2}
		Assume that $\mathrm{(A1)}$-$\mathrm{(A3)}$, $\mathrm{(H1)}$-$\mathrm{(H3)}$ hold with $(\alpha_i)_{i\in\S}$ being bounded. For $M\in \N$, let $(\La(t),\hat{\La}(t))$ be the coupling process constructed in Proposition \ref{prop-3.1}. Suppose that $\hat{Q}$ is irreducible. Let $(X^{\xi,\ell}(t),\La(t))$ be the solution to \eqref{a-1} with $X_0^{\xi,\ell}=\xi$ and $\La(0)=\ell(0)\in \S$ for $(\xi,\ell)\in \C\times \D$. Suppose that the function 
		\[i\mapsto \alpha_i+\beta_i\int_{-\tau_0}^0\!\! \e^{\hat{\alpha} \theta } \mu(\d \theta) \] is a decreasing function, where $\hat{\alpha}=\inf_{i\in \S}\alpha_i>-\infty$. Then there is $c>0$ such that for $\xi,\eta\in \C$, $\ell\in \D$, 
		\begin{align*}
			\E\big[ \|X_t^{\xi,\ell}-X_t^{\eta,\ell}\|_\infty^2\big]&\leq c\|\xi-\eta\|_\infty^2\e^{-\eta_1(\wt{Q}_M) t},\quad t>0,\\
			\E\big[\|X^{\xi,\ell}_t\|_\infty^2\big]&\leq c(1+\|\xi\|_\infty^2)\e^{-\eta_1(\wt{Q}_M)t},\quad t>0,
		\end{align*} where $\wt{Q}_M,\,\eta_1(\wt{Q}_M)$ are defined as in \eqref{e-7.5} by taking  
		$
		\theta_i=\alpha_i+\beta_i\int_{-\tau_0}^0\!\! \e^{\hat{\alpha}r} \mu(\d r)$ for $i\in\S.$
	\end{mylem}
	
	\begin{proof}
		This lemma can be proved along the same line as that of Lemma \ref{lem-2.2}.  In the notation there, the key point is the deduction of \eqref{c-5.4} from \eqref{c-5.3}. Instead of using Lemma \ref{lem-2.4}, we shall use Lemma \ref{lem-3.1} to deduce that 
		\begin{equation*}
			\begin{split}
				\E_{\p_1}\big[|\Gamma(t)|^2\big]&\leq c\|\Gamma_0\|_\infty^2 \exp\Big(\int_0^t\Big(\alpha_{\La(s+\epsilon)}\!+\!\beta_{\La(s+\epsilon)} \int_{-\tau_0}^0\!\!  \e^{\hat{\alpha}\theta}\mu(\d \theta)\Big) \d s \Big)\\
				&\leq c\|\Gamma_0\|_\infty^2\exp\Big(\int_0^t \Big(\alpha_{\hat{\La}(s+\epsilon)}\!+\!\beta_{\hat{\La}(s+\epsilon)} \int_{-\tau_0}^0\!\!  \e^{\hat{\alpha}\theta}\mu(\d \theta)\Big) \d s \Big).
			\end{split}
		\end{equation*} 
		where we used the decreasing property of the function $i\mapsto \alpha_i+\beta_i\int_{-\tau_0}^0\!\! \e^{\hat{\alpha} \theta } \mu(\d \theta)$. More details are omitted. 
	\end{proof}

	\begin{mythm}\label{main-2}
		Assume that $\mathrm{(A1)}$-$\mathrm{(A3)}$, $\mathrm{(H1)}$-$\mathrm{(H3)}$ hold with $(\alpha_i)_{i\in \S}$ being bounded. Suppose there is some $M\in \N$ such that the induced transition rate matrix $\hat{Q}_M$ defined by \eqref{e-1} is irreducible and the induced $\wt{Q}_M$ by \eqref{e-7.5} with $\theta_i=\alpha_i+\beta_i\int_{-\tau_0}^0\!\!\e^{\hat{\alpha} r}\mu(\d r)$ satisfies $\eta_1(\wt{Q}_M)>0$.  Suppose that the function $i\mapsto \alpha_i+\beta_i\int_{-\tau_0}^0\!\! \e^{\hat{\alpha} r} \mu(\d r)$  is a decreasing function, where $\hat{\alpha}=\inf_{i\in\S}\alpha_i>-\infty$. Then, there is a $c>0$ such that for any $\nu_1,\,\nu_2\in \pb_1(\C\times\D)$, for any $t>(1+\eta_1(\wt{Q}_M)/\lambda)\tau_0$, 
		\begin{equation}\label{e-10}
			\W_1(\nu_1P_t,\nu_2P_t)\leq c\Big(1+\int_{\C}\!\|\xi\|_\infty\nu_1(\d \xi,\D)\!+\!\int_{\C}\!\|\eta\|_\infty\nu_2(\d \eta,\D)\Big)\e^{-\frac{\lambda \eta_1(\wt{Q}_M)}{2(\lambda+\eta_1(\wt{Q}_M))}},
		\end{equation}
		where $\lambda\in (0, \|V\|_\infty^{-1})$ is a constant. Moreover, there is a unique invariant probability measure $\nu_\infty\in \pb_1(\C\times\D)$ such that for any $t>(1+\eta_1(\wt{Q}_M)/\lambda)\tau_0$,
		\begin{equation}\label{e-11}
			\W_1(\delta_{(\xi,\ell)}P_t,\nu_\infty)\leq c\Big(1+\|\xi\|_\infty\!+\!\int_{\C}\|\eta\|_\infty\nu_\infty(\d \eta,\D)\Big)\e^{-\frac{\lambda \eta_1(\wt{Q}_M)t}{2(\lambda+\eta_1(\wt{Q}_M))}}
		\end{equation}
		for any $\xi\in \C$ and $\ell\in \D$. 
	\end{mythm}

	\begin{proof}
		This theorem can be proved in the same way as that of Theorem \ref{main-1} by replacing $\eta_1$  with $\eta_1(\wt{Q}_M)$. The details are omitted. 
	\end{proof}
	
	\begin{myrem}\label{rem-3}
		In Lemma \ref{lem-3.2} and Theorem \ref{main-2}, the decreasing condition of $i\mapsto \alpha_i+\beta_i\int_{-\tau_0}^0\!\! \e^{\hat{\alpha}\theta} \mu(\d\theta)$ is a technical condition to simplify our statement of results. One can delete this condition by relabel the state space $\S$ so that $i\mapsto \alpha_i+\beta_i\int_{-\tau_0}^0\!\! \e^{\hat{\alpha}\theta} \mu(\d\theta)$ becomes decreasing. In such situation, one needs to redefine $Q$ and $\wt{Q}_M$ accordingly. 
	\end{myrem}
	
	Let us present a simple example on the truncation of a Markov chain on $\S=\{1,2,\ldots\}$ to a finite state space constructed in Proposition \ref{prop-3.1}. 
	\begin{myexam}\label{exam-2}
		Consider a birth-death process $(\La(t))$ on $\S=\{1,2,\ldots\}$ with $q_{ii+1}=b_i>0$ for $i\geq 1$, $q_{ii-1}=a_i>0$ for $i\geq 2$, and $q_{ij}=0$ for all $i,j$ with $|i-j|>1$, $q_i=-q_{ii}=\sum_{j\neq i} q_{ij}$. We take $M=3$ and $\S_M=\{1,2,3\}$. Then, according to the definition of \eqref{e-1}, we get
		\begin{equation*}
			\hat{Q}=\begin{pmatrix} -b_1 &b_1 &0 \\ a_2&-(a_2+b_2) & b_2\\ 0& a_3 &-a_3
			\end{pmatrix}.
		\end{equation*}
		Based on $\hat{Q}$, due to Proposition \ref{prop-3.1}, there is a coupling process $(\La(t),\hat{\La}(t))$ such that 
		\[\p\big(\hat{\La}(t)\leq \La(t),\  t\geq 0\big)=1.\]
		As a consequence, we can further apply Theorem \ref{main-2} to study the existence and uniqueness of SFDEs with Markovian regime-switching on an infinite state space. 
	\end{myexam}
	

	\noindent\textbf{Acknowledgement.} The authors are grateful to the referees for their valuable comments, which have improved the quality of this work. 
	This work is supported in part by National Key R\&D Program of China (No. 2022YFA1006000) and NNSFs of China (No. 12271397, 12531007).


\begin{thebibliography}{17}
		
		\footnotesize
		
		%
		%
		
		
		
		\bibitem{And} { \sc W.  Anderson} (1991). \emph{Continuous-time Markov chains}, Springer-Verlag, New York. 
		
		
		\bibitem{BB92} { \sc G. Basak, R. Bhattacharya} (1992). Stability in distribution for a class of singular diffusions, \emph{Ann. Probab.} 20, 312-321. 
		
		\bibitem{BM05} { \sc Y. Bakhtin,  J. Mattingly} (2005).  Stationary solutions of stochastic differential equations with memory and stochastic partial differential equations, \emph{Commun. Contemp. Math.} 7, 553-582. 
		
		\bibitem{BSY23} { \sc J. Bao, J. Shao, C. Yuan} (2023). Invariant probability measures for path-dependent random diffusions, \emph{Nonlinear Analysis} 228, 11320. 
		
		\bibitem{BGM} { \sc J. Bardet, H. Guerin, F. Malrieu} (2010). Long time behavior of diffusions with Markov switching, \emph{ALEA Lat. Am. J. Probab. Math. Stat.}  7, 151-170. 
		
		\bibitem{BS17} { \sc O. Butkovsky,  M. Scheutzow} (2017).  Invariant measures for stochastic functional differential equations, \emph{Electron. J. Probab.} 22, 1-23. 
		
		\bibitem{Chen} { \sc M. Chen} (2004). \emph{From Markov chains to non-equilibrium particle systems}, 2nd ed. Singapore: World Scientific.
		
		\bibitem{Chen05} { \sc M. Chen} (2005). \emph{Eigenvalues, inequalities, and ergodic theory}, Springer-Verlag London Limited. 
		
		\bibitem{EGS} { \sc A. Es-Sarhir, O. van Gaans,  M. Scheutzow} (2010).  Invariant measures for stochastic functional differential equations with superlinear drift term, \emph{Differential and Integral Equations}, 23, 189-200. 
		
		\bibitem{EK86} { \sc S.N. Ethier, T.G. Kurtz} (1986). \emph{Markov processes: Characterization and convergence}, John Wiley \& Sons, New York. 
		
		\bibitem{IW} { \sc N. Ikeda, S. Watanabe}  (1989).  \emph{Stochastic differential equations and diffusion processes}, North-Holland.
		
		\bibitem{LMMY} { \sc X. Li,  X. Mao, D. Mukama, C. Yuan} (2020). Delay feedback control for switching diffusion systems based on discrete time observations, \emph{SIAM J. Control Optim.} 58, 2900-2926.
		
		\bibitem{LSZ} { \sc J. Li, J. Shao, X. Zhao} (2026).  Stabilization of regime-switching processes based on discrete time observations: existence of invariant probability measure, \emph{J. Math. Anal. Appl.} 553, no. 1, 129838.
		
		\bibitem{Mao} { \sc X. Mao} (2013). Stabilization of continuous-time hybrid stochastic differential equations by discrete-time feedback control, \emph{Automatica}, 49, 3677-3681. 
		
		\bibitem{MY} { \sc X. Mao, C. Yuan} (2006).  \emph{Stochastic differential equations with Markovian switching}, Imperial College Press, London.
		
		\bibitem{NNN} { \sc D.H. Nguyen, D. Nguyen, S.L. Nguyen} (2021).  Stability in distribution of path-dependent hybrid diffusion, \emph{SIAM J. Control Optim.} 59, 434-463. 
		
		\bibitem{RRG} { \sc M. Reiß, M. Riedle, O. van Gaans} (2006). Delay differential equations driven by L\'evy processes: stationarity and Feller properties, \emph{Stoch. Process. Appl.} 116, 1409-1432. 
		
		\bibitem{Sh15c} { \sc J. Shao} (2015). Strong solutions and strong Feller propertieds for regime-switching diffusion processes in an infinite state space, \emph{SIAM J. Control Optim.} 53, 2462-2479. 
		
		\bibitem{Sh17} { \sc J. Shao} (2017). Stabilization of regime-switching processes by feedback control based on discrete time observations, \emph{SIAM J. Control Optim.} 55, 724-740. 
		
		\bibitem{Sh24} { \sc J. Shao} (2024). Comparison theorem and stability under perturbation of transition rate matrices for regime-switching processes, \emph{J. Appl. Probab.} 61, 540-557. 
		
		\bibitem{SZ21} { \sc J. Shao, K. Zhao} (2021). Continuous dependence for stochastic functional differential equations with state-dependent regime-switching on initial values, \emph{Acta Math. Sinica, Eng. Ser.} 37, 389-407. 
		
		\bibitem{SWW} { \sc B. Shi, Y. Wang, F. Wu} (2022). Ergodicity of regime-switching functional diffusions with infinite delay and application to a numerical algorithm for stochastic optimization, 
		\emph{SIAM J. Control Optim.} 60, 2650-2683.
		
		\bibitem{SZLM} { \sc G. Song, B. Zheng, Q. Luo, X. Mao} (2017). Stabilization of hybrid stochastic differential equations by feedback control based on discrete-time observations of state and mode, \emph{IET Control Theory Appl.} 11, 301-307.  
		
		\bibitem{Vil} { \sc C. Villani} (2009). \emph{Optimal transport, old and new}, in: Grundlehren der mathematischen Wissenschaften, vol. 338, Springer Berlin Heidelberg. 
		
		\bibitem{RS10} { \sc M. von Renesse, M. Scheutzow} (2010). Existence and uniqueness of solutions of stochastic functional differential equations, \emph{Random Oper. Stoch. Equ.} 18, 267-284. 
		
		\bibitem{WWY} { \sc Y. Wang, F. Wu, G. Yin} (2023). Limit theorems of additive functionals for regime-switching diffusions with infinite delay, \emph{Stoch. Process. Appl.} 167, 104215.
		
		\bibitem{WWZ} { \sc Y. Wang, F. Wu, C. Zhu} (2024). Large deviations for regime-switching diffusions with infinite delay, \emph{Stoch. Process. Appl.} 176, 104418. 
		
		\bibitem{YZ10} { \sc G. Yin,  C. Zhu} (2010).  \emph{Hybrid switching diffusions: Properties and Applications} (Stoch. Model. Appl. Prob. 63), Springer, New York.	
		
		\bibitem{YLL} { \sc S. You, W. Liu, J. Lu, X. Mao, Q. Qiu} (2015). Stabilization of hybrid systems by feedback control based on discrete-time state observations, \emph{SIAM J. Control Optim.} 53, 905-925.
		
		\bibitem{YM03} { \sc C. Yuan, X. Mao} (2003).  Asymptotic stability in distribution of stochastic differential equations with Markovian switching, \emph{Stoch. Proc. Appl.} 103, 277-291. 
		
		\bibitem{YZM} { \sc C. Yuan, J. Zou, X. Mao} (2003).  Stability in distribution of stochastic differential delay equations with Markovian switching, \emph{Systems \& Control Letters} 50, 195-207. 
		
	\end{thebibliography}
\end{document}